\documentclass[a4paper,11pt]{article}
\usepackage{amssymb}
\usepackage{amsmath}
\usepackage{amsthm}
\usepackage{graphicx}
\usepackage[T1]{fontenc} 
\usepackage[utf8]{inputenc}      	
\usepackage[english]{babel} 
\usepackage{hyperref}
\hypersetup{pdfstartview=, hidelinks} 
\usepackage{bbm} 
\usepackage{microtype} 
\usepackage{parskip}
\usepackage{centernot}

\allowdisplaybreaks 

\newcommand{\N}{\mathbb{N}} 
\newcommand{\R}{\mathbb{R}} 
\newcommand{\F}{\mathbb{F}} 
\newcommand{\FF}{\mathcal{F}} 
\newcommand{\GG}{\mathcal{G}} 
\newcommand{\Prob}{\mathbb{P}} 
\newcommand{\ExpVal}{\mathbb{E}} 
\newcommand{\PPtwo}{\mathcal{P}_2(\R^d)} 
\newcommand{\PPtwoA}{\mathcal{P}_2(\R^d \times A)} 
\newcommand{\Adm}{\mathcal{A}} 
\newcommand{\Cont}{\mathcal{C}} 
\newcommand{\LL}{\mathcal{L}} 
\newcommand{\NewOmega}{\bar{\Omega}} 
\newcommand{\Bor}{\mathcal{B}} 
\newcommand{\Prev}{\mathcal{P}(\F^0)} 
\newcommand{\CanSpace}{\Cont([0,+ \infty); \R^m)} 
\newcommand{\TestFcts}{\mathcal{TF}} 

\theoremstyle{plain}
\newtheorem{Teorema}{Theorem}[section]
\newtheorem{Lemma}[Teorema]{Lemma}
\newtheorem{Proposizione}[Teorema]{Proposition}
\newtheorem{Corollario}{Corollary}[Teorema]
\newtheorem{Definizione}[Teorema]{Definition}
\newtheorem*{Ipotesi}{Assumptions}

\theoremstyle{definition}
\newtheorem{Osservazione}[Teorema]{Remark}
\newtheorem{Esempio}[Teorema]{Example}

\numberwithin{equation}{section}

\title{Infinite Time Horizon Optimal Control of McKean-Vlasov SDEs}
\author{Silvia Rudà \thanks{Universit\`a degli Studi di Milano, Dipartimento di Matematica ``Federigo Enriques'', via Saldini 50, 20133 Milano, Italy;
		\newline
		e-mail: \texttt{silvia.ruda@unimi.it}}}
\date{}

\begin{document}
	
	\maketitle
	\begin{abstract}
		We present a theory of optimal control for McKean-Vlasov stochastic differential equations with infinite time horizon and discounted gain functional. We first establish the well-posedness of the state equation and of the associated control problem under suitable hypotheses for the coefficients. We then especially focus on the time invariance property of the value function \(V(t,\xi)\), stating that \(V\) is in fact independent of the initial time of the dynamics. This property can easily be derived if the class of controls can be restricted, forgetting the past of the Brownian noise, without modifying the value. This result is not trivial in a general  McKean-Vlasov case; in fact, we provide a counterexample where changing the class of controls worsens the value. We thus require appropriate continuity assumptions in order to prove the time invariance property. Furthermore, we show that the value function only depends on the initial random condition through its probability distribution. The function \(V\) can thus be rewritten as a map \(v\) on the Wasserstein space of order \(2\). After establishing a Dynamic Programming Principle for \(v\), we derive an elliptic Hamilton-Jacobi-Bellman equation, solved by \(v\) in the viscosity sense. Finally, using a finite horizon approximation of our optimal control problem, we prove that the aforementioned equation admits a unique viscosity solution under stronger assumptions.  
	\end{abstract}
	
		
		\vspace{5mm}


	\section{Introduction}
	\label{SectionIntroduction}
	This paper studies a general optimal control problem for McKean-Vlasov stochastic differential equations with infinite time horizon and discounted gain functional through the approach based on the Dynamic Programming Principle and Hamilton-Jacobi-Bellman equation. Fix a \(m\)-dimensional Brownian motion \(B\) and let us consider a finite-dimensional equation over the unbounded interval \((t, \infty)\) of the form
	\begin{equation*}
		dX_s= b(X_s, \Prob_{X_s}, \alpha_s)ds + \sigma(X_s, \Prob_{X_s}, \alpha_s)dB_s,
	\end{equation*}
	whose drift and diffusion coefficients (valued in \(\R^d\) and \(\R^{d \times m}\), respectively) depend not only on the state process \(X\)  itself, but also on its probability distribution and on an input control process \(\alpha\). This equation is referred to as controlled McKean-Vlasov SDE with infinite time horizon. The choice of the control \(\alpha\) is based on an optimization criterion: the representative agent aims at maximizing a gain functional of the form
	\begin{equation*}
		J(t, \xi, \alpha)= \ExpVal \Bigg[ \int_{t}^{+ \infty} e^{-\beta(s-t)} f(X_s, \Prob_{X_s}, \alpha_s) ds \Bigg]
	\end{equation*}
	over a fixed set of feasible controls, where \(f\) is a suitable running gain, which may also depend on the distribution of the state process, and \(\beta\) is a discount rate. This paper especially focuses on the properties of the value of the problem and its characterization as the unique solution to a partial differential equation (Hamilton-Jacobi-Bellman equation) on the space of probability measures with finite second moment, \(\PPtwo\), endowed with the \(2\)-Wasserstein metric.
	
	The control of McKean-Vlasov equations (also referred to as mean-field control) over a finite time interval \([0,T]\), as well as the related theory of mean-field games, was initiated independently in \cite{LasryLions} and \cite{HuangCainesMalhamé}. In these works, optimization problems of large-scale populations are considered; the agents are assumed to be symmetric and weakly interacting through the empirical measure of the population, representing its average state distribution. As these models tend to become intractable when the size of the population is large, they can be approximated by their limit behaviour, where the coefficients of the dynamics of a single representative agent depend on the probability distribution of the state at present time: if the agents are competitive, the limit model will be a mean-field game, while if they are cooperative, or subject to a central planner aiming at maximizing the collective reward, the limit model will be a mean-field control problem. We also refer the reader to \cite{Lacker2017} and to \cite{CarmonaDelarue1} and \cite{CarmonaDelarue2} for comprehensive monographs on this subject.
	
	One of the classical possible approaches to optimal control problems is based on the Dynamic Programming Principle (DPP), leading to the so-called time consistency of the problem: if an optimal control for an initial couple of data \((t,x)\) exists, then its restriction to any later time interval \([s,T]\) must be optimal for the couple \((s,X_s^{t,x})\) (see \cite{YongTimeInconsistentProblems} and references therein for a comprehensive discussion on this notion). Mean-field optimal control problems were initially thought to be time-inconsistent, due to the possibly non-linear dependence of the gain or cost functions on the law of the state processes, see e.g.\ \cite{AnderssonDjehiche} and \cite{BuckdahnDjehicheLi}, where the problem was therefore approached by means of the Pontryagin Maximum Principle, as well as in \cite{LiStochasticMaximumPrinciple}. The DPP was subsequently proved with feedback controls in \cite{PhamWei1} and \cite{PhamWei2} thanks to an appropriate version of the flow property of solutions to McKean-Vlasov SDEs and of their laws. \cite{DjetePossamaiTan} especially focuses on weak and strong formulations of the problem with associated DPPs.  	
	We also mention \cite{BayraktarCossoPhamRandomizedFilteringforMKV}, where a randomized version of the DPP was presented. In this framework, the dynamics is in fact described by a couple of SDEs (a McKean-Vlasov SDE with random initial condition and a classical SDE with deterministic initial condition and coefficients depending on the law of the solution to the first equation), as proposed in \cite{BuckdahnLiPengRainer} in an uncontrolled setting; the value function which satisfies the DPP is associated with an optimization problem involving the solution to the second state equation and thus depends on time and on the couple \((x,\pi)\), representing the initial position and law of the system.
	
	Once the DPP is established, a suitable partial differential equation (Hamilton-Jacobi-Bellman, HJB, equation) can usually be derived from it; the goal of this approach is to give an explicit characterization of the value of the problem as the unique solution to this equation, so that we can explicitly compute it for any initial time and condition, provided that we can solve HJB equation. Notice that in the mean-field case, due to the presence of a measure-valued state variable, an appropriate definition of derivative with respect to probability measures needs to be provided: in this setting, where the state at time \(s\) is in fact given by a couple of a random variable \(X_s\) and its law \(\Prob_{X_s}\), it is natural to adopt the notion of Lions derivative, introduced in \cite{LionsCours} and detailed in  \cite{Cardaliaguet2013}.
	It is well known that, except for specific cases, the value function is not regular enough to solve a PDE in the classical sense: for this reason, the value function is usually proved to solve the HJB equation only in the viscosity sense, defined by Crandall and Lions \cite{CrandallLionsUsersGuide}. For the mean-field control problem with finite time horizon, this approach has been studied in \cite{PhamWei2}, where also a verification theorem is stated, and in \cite{PhamWei1}, where the dynamics is subject to a common noise. In both papers, the value function depends on the initial random condition only through its law (i.e.\ it is law invariant) due to the specific formulation, leading to a parabolic HJB equation in the Wasserstein space of order \(2\); nonetheless, the equation is lifted to a suitable \(L^2\) space, in order that the theory of viscosity solutions in Hilbert spaces (established in \cite{LionsViscositySols1}, \cite{LionsViscositySols2} and then in \cite{FabbrGozziSwiech}) can be exploited. Unfortunately, the relation between solutions to a PDE on the Wasserstein space and solutions to its lifted version is still unclear, see Remark 3.6 in \cite{CossoUniqueness}. As a consequence, the authors of \cite{CossoExistence} presented a more general formulation of the problem, without common noise but with open loop controls, allowing even dependence of the coefficients on the entire path of the state process, valued in an infinite dimensional Hilbert space. They showed that the value function is law invariant under mild conditions and then defined the parabolic HJB equation directly on the \(2\)-Wasserstein space, rewriting it in different equivalent forms and proving that it admits a viscosity solution, corresponding to the value function. The same authors analysed uniqueness of solutions to the HJB equation in the \(2\)-Wasserstein space in \cite{CossoUniqueness} by means of a \(N\)-player smooth approximation of the original problem, based on a propagation of chaos result, and Borwein-Preiss variational principle with a suitable gauge function. Their result was subsequently generalized and corrected in \cite{CheungTaiQiu} and \cite{BayraktarCheungTaiQiu} for a second-order parabolic PDE derived from a McKean-Vlasov control problem with common noise, showing that a stronger regularity assumption on the coefficients and a different definition of viscosity supersolution are needed in order to guarantee uniqueness, see Remark 2.5 and Remark 6.1 in \cite{CheungTaiQiu}. 
	Different approaches to prove uniqueness of viscosity solutions to a first or second order PDE on \(2\)-Wasserstein spaces have recently been proposed and this is still a very active area of research, as witnessed by \cite{SonerTorus},  \cite{BayraktarFourier}, \cite{DaudinSeeger} and references therein. Notice that all the aforementioned result concern a parabolic PDE associated with a finite horizon optimal control problem.
	
	For the infinite time horizon case, ergodic mean-field control has attracted considerable interest, especially in discrete time,  where connections with reinforcement learning are possible (see e.g.\ \cite{PhamWeiDiscreteTime}, \cite{BayraktarKara} and references therein), but also in continuous time (see e.g.\ \cite{CannerozziFerrari}, where a mean-field ergodic control problem is analysed via Lagrange multipliers). On the other hand, the mean-field game counterpart of discounted optimization problems has been studied in  \cite{CarmonaInfiniteHorizon} via probabilistic techniques and in \cite{CardaliaguetPorretta} by PDE techniques  on a torus. We also mention \cite{BaoTang}, where a mean-field ergodic control problem is approached by means of an approximation by exponentially weighted infinite horizontal optimal control problems, similar to the one considered in this paper. Nonetheless, as in \cite{PhamWei1} for the finite horizon case, the elliptic HJB equations associated with the discounted control problems are lifted on a suitable \(L^2\) space, so that the Hilbert setting helps providing a uniqueness result for viscosity solutions. Moreover, only closed loop controls in feedback form are considered. 
	
	\textbf{Our contribution:} To the best of our knowledge, this is the first paper addressing a discounted optimal control problem of the McKean-Vlasov type with infinite time horizon under mild assumptions (adapted from the ones in \cite{CossoExistence}) with an analytic approach, i.e.\ through the associated HJB equation on the \(2\)-Wasserstein space. One of the main contributions of this paper is the proof of the time invariance property, stating that the value function does not depend on the initial time of the dynamics. When open-loop controls are used, this is not a trivial result, as the controls depend on the entire Brownian noise from time \(0\). We will thus need to prove that it is possible to restrict the set of admissible controls, forgetting the past of the state process, without modifying the value. Nonetheless, we will provide a motivating example showing that, in a more general McKean-Vlasov setting, the value may worsen when we optimize over the smaller class of controls; this phenomenon is strictly related to the possibility to have a non-linear dependence of the gain function with respect to the law, which can never emerge in classical control problems. We thus conclude that suitable regularity hypotheses are needed to guarantee the  invariance of the value with respect to the class of controls. In order to prove this result, we will also provide an interesting disintegration formula of a McKean-Vlasov SDE with respect to the past of the Brownian trajectory, introduced to circumvent the impossibility to apply conditional independence directly on the probability distribution, which is a deterministic object. The time invariance property is subsequently established using ``shifting'' techniques applied to the controls and to the state process. Notice that the technical issues described above are due to our choice of open loop controls; the value functions of discounted problems considered in \cite{BaoTang}, defined over a class of feedback controls, are easily proved to be time invariant. \newline
	After extending the law invariance property of the value function presented in \cite{CossoExistence} to our infinite horizon setting, we derive an elliptic partial differential equation on the Wasserstein space of order \(2\) solved by our value function. In order to prove uniqueness of solutions to this PDE, we construct a sequence of finite horizon approximations of our control problem and we rely on the aforementioned uniqueness results for viscosity solutions to parabolic PDEs:  the stronger assumptions required for the parabolic result will thus be needed even for our theorem.
	
	The paper is organized as follows. In Section \ref{SectionNotationsAssumptionsDynamicsOptimalControlProblem}, we present the notation and the assumptions, together with the dynamics and the formulation of the optimal control problem, followed by the Dynamic Programming Principle.
	Section \ref{SectionTimeInvariance} is devoted to the Time Invariance Property. We first show that, for an alternative formulation of the control problem with a smaller class of admissible controls, the desired result holds. We then introduce an example where the two formulations lead to different values. Finally, we rigorously prove that the two value function are equal under suitable assumptions, relying on the aforementioned disintegration formula for the McKean-Vlasov dynamics with respect to the first part of the Brownian trajectory; we postpone the proof of all the technical results to \ref{AppendixEquaivalenctFormulation}. In Section \ref{SectionPropertiesVF}, we present the Law Invariance Property and other essential properties of the value function in this infinite time horizon framework, from which we deduce that it can in fact be redefined as a function on the \(2\)-Wasserstein space \(\PPtwo\). Finally, Section \ref{SectionHJB} is concerned with the elliptic Hamilton-Jacobi-Bellman equation on \(\PPtwo\) and the proof that the value function is its unique viscosity solution.
	
	\section{McKean-Vlasov SDEs and optimal control}
	\label{SectionNotationsAssumptionsDynamicsOptimalControlProblem}
	\subsection{Notation and assumptions} \label{Notations}
	\subsubsection*{Space state}
	Fix two integers \(d \geq 1\), representing the dimension of the space state, and \(m \geq 1\). We denote by \( <\cdot,\cdot> \) the standard inner product in \( \R^d\) and by \( | \cdot | \) the associated Euclidean norm of a vector. Given a matrix \( M \in \R^{d \times m} \), we define its Frobenius norm \( | M | := \sqrt{\sum_{i=1}^{d} \sum_{j=1}^{m} |M_{ij}|^2}= \sqrt{tr(MM^*)} \), where \( M^* \) is the transpose of \(M\) and \( tr(\cdot) \) is the trace of a square matrix.
	\subsubsection*{Spaces of probability measures}
	Denote by \( \mathcal{P}(\R^d) \) the space of probability measures on \( (\R^d, \Bor(\R^d))\), where \( \Bor(\R^d) \) is the Borel \( \sigma\)-algebra of \( \R^d \). For any probability measure \(\mu \in \mathcal{P}(\R^d) \), we denote by \( ||\mu||_2 \) the second moment of the measure, namely \[ ||\mu||_2 := \bigg( \int_{\R^d} |x|^2 \mu(dx) \bigg)^{\frac{1}{2}}. \] Define the space
	\( \PPtwo := \{\mu \in \mathcal{P}(\R^d): || \mu ||_2 < + \infty\}. \)
	We endow it with the 2-Wasserstein metric defined as 
	\begin{align*}
		W_2(\mu,\nu):= \inf \Biggl\{ &\int_{\R^d \times \R^d} |x-y|^2 \pi(dx,dy): \pi \in \mathcal{P}_2(\R^d \times \R^d) \\& \text{ such that } \pi(\R^d \times \cdot)= \nu, \, \pi(\cdot \times \R^d)= \mu \Biggr\}^{\frac{1}{2}}
	\end{align*}
	for every \(\mu, \nu \in \PPtwo \). The space \((\PPtwo, W_2)  \) is a Polish space (see e.g.\ Theorem 6.18 in \cite{Villani}).
	
	\subsection{Probabilistic setting} \label{SectionProbabilisticSetting}
	Let \((\Omega, \FF, \Prob ) \) be a complete probability space, on which an \(m\)-dimensional Brownian motion \(B=(B_t)_{t \geq 0}\) is defined. For every random variable \(\eta\) on \((\Omega,\F,\Prob)\), we denote by \(\Prob_{\eta}\) or \(\LL(\eta)\) the law of \(\eta\). Call  \( \F^B=(\FF_t^B)_{t \geq 0}\) the \(\Prob \)- completion of the filtration generated by \(B\). Fix a sub-\(\sigma \)-algebra \( \GG \) of \( \FF  \) satisfying the following conditions:
	
	\begin{Ipotesi}[\(\boldsymbol{H_{\GG}}\)]
		\begin{enumerate}
			\item[]
			\item[(i)] \( \GG \) is the \(\sigma\)-algebra generated by a random variable \(U\) with uniform distribution on \([0,1]\), \( U \sim Unif([0,1])\).
			\item[(ii)] \(U\) and \(B\) are independent, i.e.\ \( \FF_{\infty}^B = \bigcup \limits_{t \geq 0} \FF_t^B  \) and \( \GG \) are independent.
		\end{enumerate}
	\end{Ipotesi}
	We will require initial conditions to be \(\GG\)-measurable. Notice that condition (i) guarantees that for every \(\mu \in \PPtwo\) there exists a \(\GG \)-measurable, square-integrable random variable \(\xi: \Omega \rightarrow \R^d\) such that 
	\( \mu= \Prob_{\xi} \), as shown in Lemma 2.1 in \cite{CossoExistence}. Thus, the \(\sigma\)-algebra \(\GG \) is ``sufficiently rich''. Define the filtration \( \F=(\FF_t)_{t \geq 0} \) given by \( \FF_t = \sigma( \GG \cup  \FF_t^B). \)
	
	Let \(A\) be a compact subset of a Euclidean space \(\R^N\), representing the space of control actions. Define the following space of admissible controls:
	\[ \Adm= \{ \alpha: [0, +\infty) \times \Omega \rightarrow A: \alpha \text{ is } \F \text{- predictable} \}. \]
	
	\begin{Osservazione}
		For the formulation of the problem and the properties of the value function we could relax the hypotheses on \(A\), assuming that it is a Polish space. The assumption that it has a Euclidean structure will be needed in Section \ref{SectionHJB} for the definition of viscosity  supersolution, where it allows to define the Lions derivative on \(\PPtwoA\); on the other hand, this set is required to be compact and finite dimensional for the result of uniqueness of viscosity solutions on the Wasserstein space of probability measures (see e.g.\ Remark 2.1 in \cite{CossoUniqueness}).\qed
	\end{Osservazione}
	\begin{Osservazione}
		We could equivalently define \(\Adm\) as the class of \(A\)-valued progressively measurable processes. Indeed,
		combining Lemma 3.5.(ii) and Theorem 3.7 in \cite{ChungWilliams}, for any progressive process \(\alpha\) there exists a predictable process \(\tilde{\alpha}\) such that
		\( \tilde{\alpha}(s,\omega) = \alpha(s,\omega)  \)
		for \(ds \otimes d \Prob\)-a.e. \( (s,\omega),\) generating the same trajectories and the same gain as \(\alpha\) up to null sets. \qed
	\end{Osservazione}
	Notice that the predictable \(\sigma\)-algebra has an explicit class of generators, that we will use in the following representation of the admissible controls: this Lemma is inspired by Proposition 10 in \cite{Claisse} adapted to the present framework, where the filtration corresponds to an initial enlargement of a Brownian filtration and the probability space is not necessarily canonical.
	\begin{Lemma}
		\label{DoobProcessesLemma}
		Let \( \alpha \) be a \( \F \)-predictable process valued in \(A\). Then, there exists a measurable function
		\( \underline{\alpha}: [0, +\infty) \times \CanSpace \times [0,1] \rightarrow A \) 
		such that, for \(\Prob \text{-a.e. } \omega \in \Omega,\)
		\begin{equation}
			\label{ControlExplicitForm}
			\alpha_s(\omega) = \underline{\alpha}(s, B_{. \wedge s}(\omega), U(\omega)) \qquad \forall s \geq 0, 
		\end{equation}
		where \(B_{. \wedge s}(\omega)\) denotes the trajectory of the Brownian motion \(B\) associated with \(\omega\) and stopped at time \(s\).
	\end{Lemma}
	\begin{proof}
		Denote by \( \F^{B,0}\) the filtration generated by the Brownian motion \(B\) and not completed and by \( \F^0=(\FF_t^0)_{t \geq 0} \) the filtration defined as \(\FF_t^0=\sigma(\FF_t^{B,0} \cup \GG)\) and not completed. Notice that we can restrict to \(\F^0\)-predictable controls without loss of generality. Indeed, \( \F^0 \subseteq \F \) and, by Theorem 4.37 in \cite{HeWangYan}, any stochastic process predictable with respect to \( \F \) is indistinguishable from a \(\F^0\)-predictable process. Moreover, we will only consider the almost sure set \(\NewOmega \subseteq \Omega\) of all \(\omega\) such that the associated trajectory of the Brownian motion is continuous. This is sufficient to show that our result holds for a \(\F^0\)-predictable continuous process being indistinguishable from the original control \(\alpha\).
		
		Firstly, by Doob's measurability lemma, we can prove that the thesis holds for a specific class of predictable processes of the form \(\alpha=\mathbbm{1}_C \) for some \(C \) in the \(\pi\) system \(\Cont := \mathcal{C}_1 \cup \mathcal{C}_2 \), where 
		\(  \mathcal{C}_1 = \{ [0, + \infty) \times G :  G \in \GG \} \)
		and
		~\( \mathcal{C}_2 = \{ (s, +\infty) \times (S \cap G) : s \geq 0, \, S \in \FF_s^{B,0}, \, G \in \GG \}. \) Using the Monotone Class Theorem (see e.g.\ Theorem 2.1, Chapter 0, in \cite{RevuzYor}), we extend the class of processes admitting the representation \eqref{ControlExplicitForm}, showing that it contains all the indicator processes of the form \(\mathbbm{1}_K\) for \(K \in \mathcal{K}:= \{[0,+\infty) \times G: G \in \GG\} \cup \{(t,+\infty) \times S: t \geq 0, \, S \in \FF_t^0\}. \)
		Finally, notice that, by Lemma 10.2 in \cite{Kallenberg}, the set \(\mathcal{K}\) generates the predictable \(\sigma\)-algebra \(\Prev\): using the Functional version of the Monotone Class Theorem (see e.g. Lemma 1.6 by \cite{YongZhou}), we conclude that all \(\F^0\)-predictable processes satisfy the thesis.
	\end{proof}
	Therefore, in what follows we will identify (with a slight abuse of notation) the class of admissible controls with the class of processes admitting the representation \eqref{ControlExplicitForm} for a suitable measurable function \(\underline{\alpha}\). Notice that a similar class of controls has been adopted in \cite{DeCrescenzoFuhrmanKharroubiPham}.  
	
	In the following sections, \(L^2\) stands for the Hilbert space of square-integrable random variables on \((\Omega, \FF, \Prob ) \) endowed with the norm \(|| \xi ||_{L^2}:= ( \ExpVal [ |\xi|^2] )^{\frac{1}{2}}. \)
	
	\subsection{Dynamics of the controlled system}    \label{SectionDynamics} 
	Consider the coefficients 
	\begin{align*}
		b: \R^d \times \PPtwo \times A \rightarrow \R^d \text{ and }
		\sigma: \R^d \times \PPtwo \times A \rightarrow \R^{d \times m}.
	\end{align*}
	
	Suppose that they satisfy the following assumptions:
	\begin{Ipotesi}[\(\boldsymbol{H_{b, \sigma}}\)]
		\label{HpCoefficients}
		\begin{enumerate}
			\item[]
			\item[(H1)] \(b\) and \(\sigma\) are Borel functions and, for each fixed \((x,\mu) \in \R^d \times \PPtwo\), the functions \(b(x,\mu,\cdot)\) and \(\sigma(x,\mu,\cdot)\) are continuous;
			\item[(H2)] \(b\) and \(\sigma\) are Lipschitz in \((x, \mu)\) uniformly with respect to \(a\), i.e. there exists \(  L>0 \) such that
			\begin{align*}
				&|b(x, \mu, a)- b(x', \mu', a)| + |\sigma(x,\mu,a) - \sigma(x',\mu',a)| \\
				& \leq L \big( |x-x'| + W_2(\mu, \mu') \big) \quad \forall x,x' \in \R^d, \, \mu, \mu' \in \PPtwo, a \in A;
			\end{align*}
			\item[(H3)] \(b\) and \(\sigma\) satisfy a linear growth condition in \((x, \mu)\) uniformly with respect to \(a\), i.e. there exists \( M>0 \) such that
			\begin{align*}
				|b(x, \mu, a)| + |\sigma(x, \mu, a)| \leq M(1 + |x| + ||\mu||_2)
			\end{align*}
			for all \(x,x' \in \R^d, \, \mu, \mu' \in \PPtwo, a \in A.\)
		\end{enumerate}
	\end{Ipotesi}
	\begin{Osservazione}
		The continuity in \(a\) of the functions \(b(x,\mu,\cdot)\) and \(\sigma(x,\mu,\cdot)\) is not needed to guarantee existence and uniqueness of solutions to the state equation, but we will use it in Remark \ref{RemarkEquivalentFormulationSupersolutions} to write  Hamilton-Jacobi-Bellman equation in the form \(\ExpVal[\sup_{a \in A}H(x,\mu,a)]=0\) for a suitable function \(H\). The same remark applies to the gain function \(f\) introduced in Section \ref{SectionGainFunctional}. \qed
	\end{Osservazione}
	
	Given \(\alpha \in \Adm \), for all \( t \in [0, + \infty), \, \xi \in L^2(\Omega, \FF_t, \Prob; \R^d)\), consider the state equation
	\begin{equation}
		\label{StateEquation}
		\begin{cases}
			dX_s= b(X_s, \Prob_{X_s}, \alpha_s)ds + \sigma(X_s, \Prob_{X_s}, \alpha_s)dB_s , \qquad s \in (t, \infty)\\
			X_t= \xi.
		\end{cases}
	\end{equation}
	Under hypotheses (\nameref{HpCoefficients}), Theorem 4.21 in \cite{CarmonaDelarue1} with infinite time horizon guarantees that there exists a unique (up to indistinguishability) continuous and \(\F\)-adapted solution to equation \eqref{StateEquation}, denoted by \( X^{t, \xi, \alpha} \). 
	This solution also satisfies
	\[ \ExpVal \big[ \sup \limits_{ s \in [t,T]} |X^{t, \xi, \alpha}_s|^2 \big] < + \infty \quad \text{for any } T>0. \]
	Moreover, standard arguments, relying essentially on It\^o formula and Gronwall Lemma (see for instance Theorem 1 in \cite{Pachpatte}), lead to the following estimate:
	\begin{equation} \label{InitialConditionEstimate}
		\ExpVal[| X_s^{t, \xi, \alpha}|^2] \leq \bigg(|| \xi ||_{L^2}^2 + 6M^2 (s-t) \bigg) e^{(12M^2 +1)(s-t)},
	\end{equation}
	valid for \(0 \leq t \leq s\), \(\alpha \in \Adm\) and \(\xi \in L^2(\Omega, \FF_t, \Prob)\),
	where \(M\) is the constant appearing in Assumption (H3) in (\nameref{HpCoefficients}).
	
	\subsection{Gain functional and optimization problem} \label{SectionGainFunctional}
	Given the dynamics \eqref{StateEquation}, we want to introduce an optimization problem. 
	Fix a discount factor \(\beta>0\). Define the gain function:
	\begin{align*}
		f: \R^d \times \PPtwo \times A \rightarrow \R.
	\end{align*}
	Suppose that the following assumptions are satisfied:
	
	\begin{Ipotesi}[\(\boldsymbol{H_{f, \beta}}\)]
		\label{HpGain}
		\begin{enumerate}
			\item[]
			\item[(H4)] \(f\) is a Borel measurable function and, for each fixed \((x,\mu) \in \R^d \times \PPtwo\), the function \(f(x,\mu,\cdot)\) is continuous;
			\item[(H5)] \(f\) satisfies a quadratic growth condition in \((x,\mu)\) uniformly with respect to \(a\), i.e.\ there exists \( K>0 \) such that
			\[ |f(x,\mu,a) \leq K(1+|x|^2+ ||\mu||_2^2) \quad \text{for all } x \in \R^d, \mu \in \PPtwo, a \in A.\]
			\item[(H6)] \(\beta > 18M^2 + 2 \).
		\end{enumerate}
	\end{Ipotesi}
	Define the gain functional 
	\begin{equation}
		\label{DefinitionJ}
		J(t, \xi, \alpha):= \ExpVal \Bigg[ \int_{t}^{+ \infty} e^{-\beta(s-t)} f(X_s^{t, \xi, \alpha}, \Prob_{X_s^{t, \xi, \alpha}}, \alpha_s) ds \Bigg]
	\end{equation}
	for every \( t \geq 0, \, \xi \in L^2(\Omega, \FF_t, \Prob) \) and \(  \alpha \in \Adm. \)
	
	\begin{Osservazione}
		It can be easily verified that, under hypotheses (\nameref{HpGain}), the gain functional is well defined for every \(t \geq 0\), \( \xi \in L^2(\Omega, \FF_t, \Prob) \) and \(\alpha \in \Adm\). Notice that a weaker assumption, i.e.\ \(\beta>12M^2+1\), would guarantee that the integrand appearing in \(J\) is in \(L^1(\Omega \times [t, +\infty))\), but we will need a stronger bound in Section \ref{SectionEquivalentFormulationControlProblem}. \qed
	\end{Osservazione}
	
	The following alternative set of assumptions also guarantees that the gain functional is well defined:
	\begin{Ipotesi}[\(\boldsymbol{H_{f,\beta}'}\)]
		\label{HpGainAlternative}
		\item[(H4)] \(f\) is a Borel measurable function and, for each fixed \((x,\mu) \in \R^d \times \PPtwo\), the function \(f(x,\mu,\cdot)\) is continuous;
		\item[(H5)'] \(f\) is bounded, i.e.\ there exists a constant \(C_f\) such that
		\( |f(x,\mu,a)| \leq C_f  \)
		for any \(x \in \R^d\), \(\mu \in \PPtwo\) and \(a \in A\).
		\item[(H6)'] \(\beta>0\).
	\end{Ipotesi}
	
	We introduce an additional continuity assumption on the function \(f\):
	\begin{Ipotesi}[\(\boldsymbol{H_{f}}_{cont}\)] \label{HpContf}
		The function \(f\) is locally H\"older continuous in \((x,\mu)\) uniformly with respect to \(a\), i.e.\: there exists constants \(H>0\), \(\gamma_1 \in (0,1], \gamma_2 \in (0,1]\) such that
		\begin{align}
			\label{HolderContinuityoff}
			|f(x,\mu,a)-f(x',\mu',a)| \leq H[&|x-x'|^{\gamma_1}(1+|x|+|x'|)^{2-\gamma_1}  \\
			+&W_2(\mu,\mu')^{\gamma_2} (1+||\mu||_2+||\mu'||_2)^{2-\gamma_2} ]
		\end{align}
		for every \(x,x' \in \R^d\), \(\mu,\mu' \in \PPtwo\) and \(a \in A\).
	\end{Ipotesi}
	\begin{Osservazione}
		Notice that this assumption also appears in \cite{DeCrescenzoFuhrmanKharroubiPham}; it implies that
		the function \(f\) is locally uniformly continuous in \((x, \mu)\) uniformly with respect to \(a\), as required in \cite{CossoExistence}. \qed
	\end{Osservazione}
	Our goal is the maximization of the gain functional over the class of all feasible controls. We thus define a real function \(V  \) called \textit{lifted value function} as follows:
	\begin{equation}
		\label{LiftedValueFunction}
		V(t, \xi)= \sup \limits_{\alpha \in \Adm} J(t, \xi, \alpha)
	\end{equation}
	for every \(  t \in [0,+\infty)\) and \( \xi \in L^2(\Omega, \FF_t, \Prob) \).
	
	\begin{Osservazione}
		Notice that \(V\) is not defined on a product space, as the space to which \(\xi\) belongs depends on \(t\). Moreover, we will avoid the typical measurability issues emerging in the Dynamical Programming Principle because \(V\) is directly defined on \(L^2\). \qed
	\end{Osservazione}
	
	We now state the flow property and the dynamic programming principle for the lifted value function defined in \eqref{LiftedValueFunction}, which follow from the analogous properties in the finite horizon setting (repeating the proof of Theorem 3.4 in \cite{CossoExistence}, with slight modifications).
	\begin{Lemma}[Flow Property]
		\label{FlowLemma}
		For every \(  0 \leq t \leq s \leq r < + \infty \) and for every fixed \( \alpha \in \Adm \), 
		\[ X_r^{t, \xi, \alpha} = X_r^{s,X_s^{t, \xi, \alpha},\alpha}. \] 
	\end{Lemma}
	
	\begin{Teorema}[Dynamic programming principle]
		\label{DPPThm}
		Suppose that assumptions (\nameref{HpCoefficients}) and (\nameref{HpGain}) (or (\nameref{HpGainAlternative})) hold. Then, the lifted value function satisfies the following \textbf{dynamic programming principle (DPP)}: for \(  0 \leq t \leq s < + \infty \) and for any \(  \xi \in L^2(\Omega, \FF_t, \Prob) \)
		\begin{align} 
			\label{DPP}
			V(t, \xi)= \sup \limits_{\alpha \in \Adm} \biggl\{ \ExpVal \bigg[ \int_{t}^{s} &e^{-\beta(r-t)} f(X_r^{t, \xi, \alpha}, \Prob_{X_r^{t, \xi, \alpha}}, \alpha_r) dr \bigg] \nonumber \\+& e^{-\beta(s-t)} V(s, X_{s}^{t, \xi, \alpha}) \biggr\}.
		\end{align}
	\end{Teorema}

	\section{The Time Invariance Property}
	\label{SectionTimeInvariance}
	\subsection{Preliminary considerations}
	\label{SectionTimeInvarianceForVaryingInTimeValueFunction}
	In general, when the time horizon is not finite and the coefficients are autonomous, the value function of a discounted optimal control problem does not depend on time: intuitively, if the time available to realise the gain is not bounded, the maximum attainable gain should not be affected by the initial time instant. In this section we are proving that this property holds even in the McKean-Vlasov case for the function \(V(t,\xi)\).
	
	As a preliminary step, for any starting time  \(t \geq 0\) we introduce an  alternative class of controls, denoted by \(\Adm_t\), and the corresponding value function \(V_t(t,\xi)\).  Denote by \(B^{(t)}=(B_s^{(t)})_{s \geq t}\) the process given by the increments of \(B\) from time \(t\), i.e.\ \(B_s^{(t)}=B_s-B_t\) for every \(s \geq t\). Notice that \(B^{(t)}\) is a Brownian motion on \([t, +\infty)\). We set 
	\( \Adm_t:= \{ \alpha: [t,+ \infty) \times \Omega \rightarrow A: \alpha \text{ is } \F^t \text{-predictable} \},  \)
	where \(\F^t\) is the filtration \(\F^t=(\FF_s^t)_{s \geq t}\) such that \(\FF_s^t= \sigma(\FF_s^{B,t} \cup \GG) \) and \(\F^{B,t}\) is the completed Brownian filtration of  \(B^{(t)}\). Intuitively,   
	\( \Adm_t\) can be interpreted as the class of controls from time \(t\)  that ``forgets'' the past of the Brownian motion. Choosing \( \Adm_t\)  instead of \( \Adm\)  from the beginning  is  more natural from a modelling viewpoint but may generate difficulties in the proof of the  Dynamic Programming Principle.

	Notice that \(\Adm_t \subseteq \Adm\) for any \(t \geq 0\) and \(\Adm_0= \Adm\).  Define
	\begin{equation}
		\label{DefinitionVt}
		V_t(t,\xi):= \sup_{\alpha \in \Adm_t} J(t,\xi,\alpha).
	\end{equation}
	It is thus clear that \(V_t(t,\xi) \leq V(t,\xi)\) and  \(V_t(0,\xi) = V(0,\xi)\). For this value function  we can easily prove the following result.
	\begin{Proposizione}
		\label{VaryinginTimeTimeInvarianceProperty}
		Suppose that assumptions (\nameref{HpCoefficients}) and (\nameref{HpGain}) or (\nameref{HpGainAlternative}) are satisfied. Let \(\xi \in L^2(\Omega, \GG, \Prob).\) Then,
		\begin{equation}
			\label{VaryinginTimeEquationTimeInvariance}
			V(0, \xi)=V_t(t,\xi) \qquad \forall t \geq 0.
		\end{equation}
	\end{Proposizione}
	The proof relies on the following lemma.
	\begin{Lemma}
		\label{StationaritySDEs}
		Consider two solutions of the same controlled stochastic equation of McKean-Vlasov type \eqref{StateEquation} starting at different initial times from the same initial condition and under different controls. We denote by \(X^{(0)}=(X_s^{(0)})_{s \geq 0}\) the unique solution to equation
		\begin{equation}
			\label{StateEquationX0}
			\begin{cases}
				dX_s^{(0)}=b(X_s^{(0)},\Prob_{X_s^{(0)}},\alpha_s)ds+ \sigma(X_s^{(0)},\Prob_{X_s^{(0)}},\alpha_s)dB_s \quad \forall s \geq 0&\\
				X_0^{(0)}= \xi&
			\end{cases}
		\end{equation}
		and by \(X^{(t)}=(X_s^{(t)})_{s \geq 0}\) the unique solution to
		\begin{equation*}
			\begin{cases}
				dX_s^{(t)}=b(X_s^{(t)},\Prob_{X_s^{(t)}},\gamma_s)ds+ \sigma(X_s^{(t)},\Prob_{X_s^{(t)}},\gamma_s)dB_s \quad \forall s \geq t&\\
				X_t^{(t)}= \xi, &
			\end{cases}
		\end{equation*}
		where the coefficients satisfy Assumptions (\nameref{HpCoefficients}) and \(\alpha \in \Adm_0\), \(\gamma \in \Adm_t\). Assume furthermore that
		\begin{align}\label{SameLawsHp} 
			\LL(\xi,(B_s)_{s \geq 0},(\alpha_s)_{s \geq 0}) = \LL(\xi,(B_{s+t}^{(t)})_{s \geq 0},(\gamma_{s+t})_{s \geq 0}). 
		\end{align}
		Then, 
		\[ \LL((X_s^{(0)})_{s \geq 0}, (\alpha_s)_{s \geq 0})= \LL((X_{s+t}^{(t)})_{s \geq 0}, (\gamma_{s+t})_{s \geq 0}). \]
	\end{Lemma}
	\begin{proof}
		Using a deterministic change of time in the stochastic integrals, it can be verified that the process \( (X_{s+t}^{(t)})_{s \geq 0}\) solves equation \eqref{StateEquationX0} on \([0, +\infty)\) with noise   \((B_{s+t}^{(t)})_{s \geq 0}\) in place of \(B\) and control \((\gamma_{s+t})_{s \geq 0}\) instead of \(\alpha\). The conclusion then follows from \eqref{SameLawsHp} by standard arguments: for instance, the desired equality in law can be proved for the Picard iterates used to construct the solutions, e.g.\ following the lines of the proof of  Theorem 9.5 in \cite{Baldi}. 
	\end{proof}

	\begin{proof}[Proof of Theorem \ref{VaryinginTimeTimeInvarianceProperty}] Intuitively, the proof relies on the idea to ``shift'' the controls forward and backward in time using \(B^{(t)}\) instead of \(B\) (and vice versa) in expression \eqref{ControlExplicitForm}, leading to a bijection from \(\Adm\) to \(\Adm_t\).
		
		Fix a \(\F\)-predictable control \(\alpha_s=\underline{\alpha}(s,B_{\cdot \wedge s}, U)\), \(s \geq 0.\)
		Choose \(t \geq 0\). Consider a measurable function \(\bar{\alpha}^t\) such that
		\[  \bar{\alpha}^t(z,w',u)=\underline{\alpha}(z-t, w'_{\cdot \wedge z-t },u) \quad \text{ for } z \in [t,+\infty), w \in \Cont([0,+\infty);\R^m), u \in [0,1],\]
		and construct the \(\F^t\)-predictable control \(\bar{\alpha}:=\bar{\alpha}^t(\cdot,(B_{s+t}^{(t)})_{s \geq 0},U)\).
		Therefore, we can denote by \(\bar{X}=X^{t, \xi, \bar{\alpha}}\) the unique solution to the state equation with initial time \(t\) and initial state \(\xi\) under the feasible control \(\bar{\alpha}\), i.e.\ equation
		\begin{equation}
			\label{ProofTimeInvarianceEq3}
			\begin{cases}
				d\bar{X}_s= b(\bar{X}_s,\Prob_{\bar{X}_s}, \bar{\alpha}^t(s,B_{\cdot \wedge s+t}^{(t)},U) ) ds &\\ \qquad \, + \sigma(\bar{X}_s,\Prob_{\bar{X}_s}, \bar{\alpha}^t(s,B_{\cdot \wedge s+t}^{(t)},U) ) d B_s &\\
				\bar{X}_t = \xi. &
			\end{cases}
		\end{equation}
		Notice that, by construction, \( \LL((B_s)_{s \geq 0},U,(\alpha_s)_{s \geq 0}) = \LL((B_{t+s}^{(t)})_{s \geq 0},U,(\bar{\alpha}_{s+t})_{s \geq 0})\), which implies 
		\( \LL((B_s)_{s \geq 0},\xi,(\alpha_s)_{s \geq 0}) = \LL((B_{t+s}^{(t)})_{s \geq 0},\xi,(\bar{\alpha}_{s+t})_{s \geq 0})))  \)
		for a \(\sigma(U)\)-measurable \(\xi\). As a consequence, applying Lemma \ref{StationaritySDEs}, we can affirm that
		\[ \LL((X_s^{0,\xi,\alpha})_{s \geq 0}, (\alpha_s)_{s \geq 0}) = \LL((\bar{X}_{s+t})_{s \geq 0},(\bar{\alpha}_{s+t})_{s \geq 0}) . \]
		Computing the gain, by the previous equality in law and a simple change of variable \(z=s+t\),
		\begin{equation*}
			J(0,\xi,\alpha)= \ExpVal \Bigg[  \int_{t}^{+ \infty} e^{-\beta (z-t)} f(\bar{X}_z, \Prob_{\bar{X}_z}, \bar{\alpha}_z) dz\Bigg] = J(t,\xi,\bar{\alpha}) \leq V_t(t,\xi).
		\end{equation*}
		Taking the supremum over all feasible controls, we conclude that
		\( V(0, \xi) \leq V_t(t,\xi). \)
		
		We now want to prove the converse inequality. Fix \(\gamma \in \Adm_t\). By Lemma \ref{DoobProcessesLemma}, on  \([t, +\infty)\), 
		\(\gamma\) is indistinguishable from
		\( (\underline{\gamma}(s,B^{(t)}_{\cdot \wedge s},U))_{s \geq t}  \)
		for a suitable measurable map \(\underline{\gamma}:[t, +\infty) \times \Cont([t, +\infty)) \times [0,1]\). Defining  a control \(\alpha \in \Adm\) setting 
		\( \alpha_s:= \underline{\gamma}(s+t, B_{. \wedge s+t}, U) \) for all \(s \geq 0\), similar arguments lead to
		\begin{equation*}
			J(t, \xi, \gamma) = J(0, \xi, \alpha) \leq V(0,\xi).
		\end{equation*} 
		Taking the supremum over all \(\gamma \in \Adm_t\), we conclude that
		\( V_t(t, \xi) \leq V(0,\xi).  \)
	\end{proof}
	
	\subsection{A counterexample to the time invariance property}
	\label{Sectioncounterexample}
	
	Contrary to the case of classical control theory, time invariance for the value function \(V(t,\xi)\) of a McKean-Vlasov control problem may fail. Thanks to Proposition \ref{VaryinginTimeTimeInvarianceProperty}, it is enough to exhibit an example where \(V(t,\xi)\neq V_t(t,\xi)\) for some \(t>0\). 
	We provide here an example of a finite horizon optimal control problem of the McKean-Vlasov type where restricting the class of controls from \(\Adm\) to \(\Adm_t\) worsens the optimal value.

	\begin{Esempio} \label{CounterExample}
		Fix a complete probability space \((\Omega, \FF, \Prob)\) on which a scalar Brownian motion \(B\) is defined. Consider a deterministic initial condition \(\xi=0\). Fix the dimension \(d=1\), the time horizon \(T=1\) and the action space \(A=[-1,1]\). Choose an initial time \(t \in (0, \frac{1}{2})\). Consider the following state equation:
		\begin{equation}
			\label{ExampleStateEquation}
			dX_s= \alpha_s ds, \quad s \geq t, \quad X_t=0.
		\end{equation}
		
		\begin{Osservazione}
			Except for the adaptedness condition on \(\alpha\), equation \eqref{ExampleStateEquation} looks like a deterministic controlled equation. Nonetheless, we may regard it as a stochastic system as follows. Choose  \(R\) large enough to guarantee that  \(\sup \limits_{s \in [t,T]} |X_s^{t,0,\alpha}| \leq 1-t < R  \) and take
			\(b:\R \rightarrow \R\) and  \(\sigma: \R \rightarrow \R\) Lipschitz with  \(b(x)=\sigma(x)=0\) for \(|x|<R\).
			The trajectories of 
			\eqref{ExampleStateEquation} thus coincide with the trajectories of the stochastic controlled system
			\begin{equation*}
				dX_s= (\alpha_s + b(X_s)) ds + \sigma(X_s) dB_s, \quad s \geq t, \quad X_t=0.
			\end{equation*}
			\qed
		\end{Osservazione}

		Define the following cost functional of the McKean-Vlasov type, which we want to minimize over a suitable class of controls: : 
		\begin{equation}
			\label{ExampleCostFunctional}
			J(t,0,\alpha):=H(\Prob_{X_1^{t,0,\alpha}}|\hat{\nu}),
		\end{equation}
		where \(\hat{\nu}=\frac{1}{2} \delta_{\{1-t\}}+ \frac{1}{2} \delta_{\{t-1\}}\). The function \(H\) is the modified relative entropy, defined as in Definition 15.1 in \cite{AmbrosioBruéSemola} and in Definition 9.4.1 in \cite{AmbrosioGigliSavaré}: \begin{equation}
			\label{DefinitionRelativeEntropy}
			H(\mu|\nu):=
			\begin{cases}
				Q \quad \text{if } \mu \centernot{\ll} \nu &\\
				\int_{\R^d} h \bigg(\frac{d \mu}{d \nu} (x) \bigg) \nu(dx) \quad \text{if } \mu \ll \nu&
			\end{cases}
		\end{equation}
		where \(Q\) is a sufficiently large real value and \(h: [0,+\infty) \rightarrow [0,+\infty)\) is a continuous, strictly convex, non-negative function having a unique minimum at \(s=1\) (e.g.\ \(h(y)=y \log(y)+(1-y)\) with \(h(0)=1\)) and \(\frac{d \mu}{d \nu}\) denotes the density of \(\mu\) with respect to \(\nu\) when it exists. Notice that \(H\) has a  unique strict minimum at \(\mu=\hat\nu\). Indeed, when  \(\mu \ll \hat{\nu}\) then \(\mu\) is supported in \(\{1-t, t-1\}\) with density 
		\begin{equation}
			\label{ExampleDensityofMuwrthatNu}
			\frac{d \mu}{d \hat{\nu}}(y)= 2 \lambda \mathbbm{1}_{\{1-t\}}(y) + 2(1-\lambda) \mathbbm{1}_{\{t-1\}}(y)
		\end{equation}
		for some \(\lambda \in [0,1]\).	
		Therefore, by the convexity of \(h\),
		\begin{align}
			H(\mu|\hat{\nu})
			&= \frac{1}{2} [h( 2 \lambda)+h(2(1-\lambda))] \nonumber \\
			&\geq  h \bigg(\frac{1}{2} (2 \lambda) +  \frac{1}{2} (2 (1-\lambda)) \bigg) =h(1) =H(\hat\nu|\hat\nu),\qquad \lambda \in [0,1].
			\label{ExampleWhereIUseStrictConvexity}
		\end{align} As \(h\) is strictly convex, the inequality is strict for any \(\lambda \ne \frac{1}{2}\),  proving that  the global minimum of \(H\) is reached only at \(\hat{\nu}\). 
		
		\textbf{1)} We first minimize the cost functional over the class \(\Adm_t\).
		
		Call \(\alpha_1\) and \(\alpha_2\) the controls which are identically equal to \(1\) and \(-1\), respectively.  As \(T=1\),
		\( X_1^{t,0,\alpha}= \int_{t}^{1} \alpha_s ds \)
		with \(\alpha_s \in [-1,1]\) for every \(s \in [0,1]\). Therefore, \(\Prob_{X^{t,0,\alpha_1}}=\delta_{1-t}\), \(\Prob_{X^{t,0,\alpha_2}}=\delta_{t-1}\) and for any other \(\alpha \in \Adm_t\)   \(\Prob_{X_1^{0,\xi,\alpha}}\) is not supported in \(  \{1-t, t-1\}\). This implies that both  \(\alpha_1\) and \(\alpha_2\) are optimal with value 
		\(({h(0)+h(2)})/{2}\) by 
		\eqref{ExampleWhereIUseStrictConvexity}.

		\textbf{2)} We now minimize the cost functional over the class  \(\Adm\).
		
		Consider the control \(\alpha^*\) defined as follows:
		\begin{equation}
			\label{ExampleDefinitionAlphaOpt}
			\alpha_s^*(\omega) := 
			\begin{cases}
				1 \qquad \text{if } B_{\frac{t}{2}}(\omega) \geq 0 &\\
				-1 \qquad \text{if } B_{\frac{t}{2}}(\omega) < 0 &
			\end{cases}
		\end{equation}
		for every \(s \geq t, \omega \in \Omega\). Clearly, \(\alpha^* \in \Adm \setminus \Adm_t\) because it depends on the source of randomness \(B_{\frac{t}{2}}\),  independent of \(B^{(t)}\). Moreover, we can rewrite it as 
		\( \alpha^*= \underline{\alpha}^*(B_{. \wedge t}), \)
		where  \(\underline{\alpha}^*(w):= \alpha_1\) if \(w_{\frac{t}{2}} \geq 0\) and \(\underline{\alpha}^*(w):= \alpha_2\) if \(w_{\frac{t}{2}} < 0\). Then
		\begin{equation}
			\label{ExampleLinearCombinationofTwoLaws}
			\Prob_{X_1^{t,0,\alpha^*}}= \lambda \Prob_{X_1^{t,0,\alpha_1}} + (1-\lambda) \Prob_{X_1^{t,0,\alpha_2}},
		\end{equation}
		
		by conditional independence, where \( \lambda=\Prob(\underline{\alpha}^*(B_{\cdot \wedge t})= \alpha_1) = \Prob(B_{\frac{t}{2}} \geq 0)=\frac{1}{2} \). 
		Therefore  \(
		\Prob_{X_1^{t,0,\alpha^*}}= \hat{\nu}\) and from  \eqref{ExampleWhereIUseStrictConvexity} we deduce that \(H(\Prob_{X_1^{t,0,\alpha^*}}|\hat{\nu})=h(1),\)
		which is the unique global minimum of \(H\).
		
		By  \textbf{1)} and \textbf{2)} we conclude that
		\begin{equation*}
			\inf \limits_{\alpha \in \Adm_t} H(\Prob_{X_T^{t,0,\alpha}}|\hat{\nu}) = \frac{h(2)+h(0)}{2} > h(1) = \inf \limits_{\alpha \in \Adm} H(\Prob_{X_T^{t,0,\alpha}}|\hat{\nu}):
		\end{equation*}
		the value at time \(t\) is strictly larger if we optimize over the smaller class of controls \(\Adm_t\).
	\end{Esempio}
	
	\begin{Osservazione}
		The conclusion above follows from the strict convexity of \(h\), used in  \eqref{ExampleWhereIUseStrictConvexity}, i.e.\ from the non-linear dependence of the gain functional on the terminal law of the state process. Therefore, it is not possible to construct an analogous example in classical control theory, where the dependence of the terminal gain upon the law is always linear. \qed
	\end{Osservazione}

	\subsection{Time invariance of the value function}
	\label{SectionEquivalentFormulationControlProblem}
	
	The previous example shows that appropriate conditions are required to guarantee the equality \(V(t,\xi)=V_t(t,\xi)\) and hence the time invariance property.  We will prove the following:
	\begin{Teorema} \label{ThmValueOverSmallerClassControls}
		Assume that (\nameref{HpCoefficients}), (\nameref{HpGain}) (or (\nameref{HpGainAlternative})) and (\nameref{HpContf}) are satisfied. Then, for every choice of \(t \geq 0\) and \(\xi \in L^2(\Omega,\GG, \Prob)\), it holds that
		\begin{equation} \label{EqualityWRTSmallerClassofControls}
			V_t(t, \xi) = \sup \limits_{\alpha \in \Adm_t} J(t,\xi,\alpha) = \sup \limits_{\alpha \in \Adm} J(t,\xi,\alpha) =V(t, \xi).
		\end{equation}
	\end{Teorema}

	The continuity assumption (\nameref{HpContf}), which was missing  in Example \ref{CounterExample}, plays a crucial role here. 
	In the proof of Theorem \ref{ThmValueOverSmallerClassControls}, it guarantees the continuity of the gain functional with respect to time, but it also qualifies as a key technical assumption for the rest of the paper (especially for the Law Invariance Property and the continuity of the value function).
	In order to prove Theorem \ref{ThmValueOverSmallerClassControls}, we need to introduce the notations and technical results contained in the following paragraph.
	
	\subsubsection{Disintegration of McKean-Vlasov systems}
	
	We first recall the definition of concatenation of paths.
	\begin{Definizione}
		Fix a time instant \(t \geq 0\). Let \(w_1, w_2\) be two \(\R^m\)-valued functions with domain \([0+\infty)\) and \([t, +\infty)\), respectively. Define the concatenation of \(w_1\) and \(w_2\) as the  function \(w_1 \otimes_t w_2\) such that
		\[ (w_1 \otimes_t w_2)(s) =
		\begin{cases}
			w_1(s) \qquad& \text{ if } 0 \leq s \leq t \\
			w_1(t) + w_2(s) -w_2(t) \qquad& \text{ if } s \geq t. 
		\end{cases} \]
	\end{Definizione}
	
	In particular, if \(w_1\) and \(w_2\) are Brownian trajectories, we deduce that, for every \(t \geq 0\),
	\[(B \otimes_t B)(\omega)=(B \otimes_t B^{(t)})(\omega)= B(\omega).\] 
	As a consequence, by Lemma \ref{DoobProcessesLemma}, every \(\F\)-predictable process \(\alpha\) valued in \(A\) is indistinguishable both from \((\underline{\alpha}(s,(B \otimes_t B)_{\cdot \wedge s}, U))_{s \geq 0}\) and from  \((\underline{\alpha}(s,(B \otimes_t B^{(t)})_{\cdot \wedge s}, U))_{s \geq 0}\) for a suitable measurable function \( \underline{\alpha}:[0,+\infty) \times \CanSpace \times [0,1] \rightarrow A \), for any choice of \(t \geq 0\).
	
	We now want to state a ``disintegration property'' of the Mckean-Vlasov equation with respect to the first part of the Brownian trajectory. Assume that hypotheses (\nameref{HpCoefficients}) and (\nameref{HpGain}) (or (\nameref{HpGainAlternative})) are satisfied. Fix an initial time \(r>0\) and a control \(\alpha=(\underline{\alpha}(s, B_{. \wedge (r \wedge s)} \otimes_r B_{.\wedge s}^{(r)},U))_{s \geq 0} \in \Adm.\)
	Call \(X:=X^{r,\xi,\alpha}\) the unique solution to the McKean-Vlasov SDE \eqref{StateEquation} starting from \(\xi\) at time \(r\) under the control \(\alpha\), i.e.\
	\begin{equation}
		\label{StateEquationFromr}
		dX_s = b(X_s,\Prob_{X_s}, \alpha_s) ds + \sigma(X_s,\Prob_{X_s}, \alpha_s) dB_s, \quad s \geq r, \quad X_r=\xi.
	\end{equation}
	Now associate to each path \(w \in \Cont([0,r])\) a control \(\alpha^r(w) \in \Adm_r\) defined as follows:
	\begin{equation}
		\label{Definitionalpharw}
		\alpha_s^r(w):= \underline{\alpha}(s, w \otimes_r B_{\cdot \wedge s}^{(r)},U) \quad \text{ for all } s \geq r.
	\end{equation}
	Consider the following system of SDEs:
	\begin{align}
		\label{DisintegrationEquation}
		dX_s^w&= b \bigg(X_s^w, \int_{\Cont([0,r])} \Prob_{X_s^{w'}} W_{[0,r]}(w'), \alpha_s^r(w)\bigg) ds \nonumber \\
		&+ \sigma \bigg(X_s^w, \int_{\Cont([0,r])} \Prob_{X_s^{w'}} W_{[0,r]}(w'), \alpha_s^r(w) \bigg) dB_s, \quad s \geq r, \quad X_r^w=\xi
	\end{align}
	parametrized by \(w \in \Cont([0,r])\),  where \(W_{[0,r]}\)  denotes the Wiener measure on the space \(\Cont([0,r])\).
	\begin{Lemma} \label{LemmaWellPosednessofDisintegrationEquation}
		Under hypotheses (\nameref{HpCoefficients}), the system of SDEs \eqref{DisintegrationEquation} is well-posed, i.e.\ it admits a unique solution \((X^w)_{w \in \Cont([0,r])}\).
	\end{Lemma}
	
	We postpone the proof to Appendix \ref{AppendixEquaivalenctFormulation}. 
	Letting  \(X=(X^w)_{w \in \Cont([0,r])}\) denote the solution to   \eqref{DisintegrationEquation}, call \(X^B:=X_{|w=B_{. \wedge r}}^w\) the stochastic process defined by   composition 
	with the Brownian trajectory on   \([0,r]\), namely
	\[  X_s^B(\omega)= X_s^{B_{[0,r]}(\omega)}(\omega) \]
	for every \(s \geq r,\) \(\omega \in \Omega\).  Then we have the following disintegration result.

	\begin{Proposizione} 
		\label{LemmaDisintegrationandCompositionBgivesX}
		Under hypotheses (\nameref{HpCoefficients}),  up to null sets, \(X^B\) is equal to the unique solution \(X^{r,\xi,\alpha}\) to the state equation \eqref{StateEquationFromr} starting from \(\xi\) at time \(r\) under the control \(\alpha\).
	\end{Proposizione}

	\begin{proof} 
		We begin by observing that \(X^B\) is well defined because, by Section 12 in \cite{StrickerYorParamètre}, the map
		\( (s,\omega,w) \rightarrow X_s^w(\omega)  \)
		is measurable. Moreover, for every fixed \(w \in \Cont([0,r])\), \(X^w\) is independent of \((B_s)_{s \in [0,r]}\), as all the sources of randomness appearing in equation \eqref{DisintegrationEquation} (\(\xi\), \(U\) and \(B^{(r)}\)) are independent of \((B_s)_{s \in [0,r]}\). As \(X=(X^w)_{w \in \Cont([0,t])}\) is the solution to  \eqref{DisintegrationEquation}, \(X^B\) solves
		\begin{align}
			\label{EquationXB}
			dX_s^B&=b \bigg(X_s^B, \int_{\Cont([0,r])} \Prob_{X_s^{w'}} W_{[0,r]}(dw'), \alpha_s^r(w)_{|w=B_{[0,r]}} \bigg) ds \nonumber \\
			&+ \sigma \bigg(X_s^B, \int_{\Cont([0,r])} \Prob_{X_s^{w'}} W_{[0,r]}(dw'), \alpha_s^r(w)_{|w=B_{[0,r]}} \bigg) dB_s, \quad s \geq r \nonumber\\
			&X_r^B=\xi.
		\end{align}
		We want to prove that equation \eqref{EquationXB} corresponds to equation \eqref{StateEquationFromr}; in this case, the thesis follows from uniqueness of solution to McKean-Vlasov equations.
		
		Notice that, by definition \eqref{Definitionalpharw} of \(\alpha^r(w)\), 
		\( \alpha_s^r(w)_{|w=B_{[0,t]}}=\alpha_s \)
		for every \(s \geq r\). It is thus sufficient to prove that
		\[ \Prob_{X_s^B}= \int_{\Cont([0,r])} \Prob_{X_s^{w'}} W_{[0,r]}(dw') \quad \text{for every } s \geq r.  \]
		Now, by Fubini theorem, for every \(s \geq r\) and every \(\phi \in \Cont_b(\R^d)\),
		\begin{align*}
			\int_{\R^d} \phi(z)  \bigg(\int_{\Cont([0,r])} \Prob_{X_s^{w'}} W_{[0,r]}(dw')\bigg)(dz)
			&= \int_{\Cont([0,r])}  \ExpVal[\phi(X_s^{w'})] W_{[0,r]}(dw') \\
			&=\ExpVal \bigg[ \phi(X_s^{B}) \bigg] = \int_{\R^d} \phi(z) \Prob_{X_s^{B}}(dz)
		\end{align*}
		by conditional independence (as \(X^w\) is independent of \(B_{[0,t]}\))
		and the thesis is satisfied.
	\end{proof}

	\begin{Osservazione}
		Notice that, under the assumptions of Lemma \ref{LemmaWellPosednessofDisintegrationEquation}, by Fubini Theorem 
		\begin{align}
			\biggl| \biggr|\int_{\Cont([0,r])} \Prob_{X_v^{w'}} W_{[0,r]}(dw')\biggl| \biggr|_2^2 &= \int_{\Cont([0,r])}  ||\Prob_{X_v^{w'}} ||_2^2 W_{[0,r]}(dw') = \nonumber \\ &\int_{\Cont([0,r])} \ExpVal [|X_v^{w'}|^2] W_{[0,r]}(dw')
			\label{SecondMomentofIntP}
		\end{align}
		As a consequence, applying Lemma \ref{LemmaDisintegrationandCompositionBgivesX} in \eqref{SecondMomentofIntP} and using estimate  \eqref{InitialConditionEstimate}, we deduce that
		\begin{align} \label{EstimateforXw}
			\ExpVal[ |X_s^w|^2 ] \leq  \bigg(h(s, ||\xi||_{L^2}^2)+6M^2 \int_{r}^{s} h(s, ||\xi||_{L^2}^2) e^{(12M^2 +1)(v-r)} dv\bigg) e^{(6M^2+1)(s-r)}
		\end{align}
		for a suitable function \(h\) linear and increasing in both arguments.
		If \(f\) and \(\beta\) satisfy (\nameref{HpGain}), for every fixed \(w \in \Cont([0,r])\) we can introduce the following gain functional:
		\begin{equation}
			\label{DefinitiontildeJ}
			J^w(r,\xi,\alpha):= \ExpVal \bigg[ \int_{r}^{+ \infty} e^{-\beta(s-r)} f(X_s^{w}, \int_{\Cont([0,r])} \Prob_{X_s^{w'}} W_{[0,r]}(dw'), \alpha_s^r(w)) dr  \bigg],
		\end{equation}
		which is well defined by standard estimates, relying essentially on It\^o formula for functions of probability measures (see Theorem 4.16 and Remark 4.17 in \cite{CossoExistence}) and on Lemma \ref{LemmaDisintegrationandCompositionBgivesX}, provided that \(\beta > 18M^2 + 2\).
		On the other hand, if \(f\) satisfies (\nameref{HpGainAlternative}), then \(J^w\) is well defined for any \(\beta>0\). \qed
	\end{Osservazione}

	\subsubsection{Proof of the time invariance property}
	
	In order to prove Theorem \ref{ThmValueOverSmallerClassControls} we
	need to establish some  continuity property of the gain functional \(J\).
	
	\begin{Proposizione}[Right Continuity of \(J\)]
		\label{ThmContinuityofJwrtTime}
		Assume that hypotheses (\nameref{HpCoefficients}), (\nameref{HpGain}) (or (\nameref{HpGainAlternative})) and (\nameref{HpContf}) are satisfied. Fix an initial condition \(\xi \in L^2(\Omega,\GG, \Prob)\) and a time \(t \geq 0\). Then, the function \((u, \alpha) \rightarrow J(u,\xi,\alpha)\)
		is right continuous at \(t\) in the variable \(u\) uniformly with respect to \(\alpha \in \Adm\), i.e.\ for every \(\epsilon>0\) there exists a \(\zeta=\zeta(t, \epsilon,\xi)\) such that, if \(u \in (t, t+\zeta)\), then
		\[ |J(t,\xi,\alpha)-J(u,\xi,\alpha)| < \epsilon.  \]
	\end{Proposizione}
	
	We postpone the proof  to Appendix \ref{AppendixEquaivalenctFormulation} and we present here the proof of the main result of this section.

	\begin{proof}[Proof of Theorem \ref{ThmValueOverSmallerClassControls}]
		Fix an initial condition \(\xi \in L^2(\Omega, \GG, \Prob)\) and a time \(t \geq 0\). First notice that \(\Adm_t \subseteq \Adm\); therefore,
		\[ V_t(t, \xi) \leq V(t,\xi).  \] 
		For the converse inequality, consider a time \(u > t\). By Proposition \ref{ThmContinuityofJwrtTime}, we can say that for any \(\epsilon>0\) there exists a \(\delta>0\) such that
		\begin{equation}
			\label{ProofThmSmallerClassControls2}
			|J(t,\xi,\alpha')-J(u,\xi,\alpha')|< \epsilon \quad \forall \alpha' \in \Adm,
		\end{equation}
		provided that \(u \in (t, t+\delta)\). As \(\Adm_t \subseteq \Adm\), \eqref{ProofThmSmallerClassControls2} clearly holds even if we substitute \(\Adm\) with \(\Adm_t\). Assume for the moment that the following claim holds: 
		\begin{align}
			\label{ProofThmSmallerClassControlsCLAIM}
			\text{For every } u>t \text{ and every fixed } &\alpha \in \Adm  \text{ there exists a control }  \alpha'=\alpha'_{\alpha,u} \in \Adm_t \nonumber \\
			\text{ such that } &J(u,\xi,\alpha')=J(u,\xi,\alpha).
		\end{align}
		Now fix a time instant \(u \in (t, t+ \delta)\) (with \(\delta\) as in equation \eqref{ProofThmSmallerClassControls2}); for a fixed control \(\alpha \in \Adm\) choose a control  \(\alpha'=\alpha'_{\alpha,u} \in \Adm_t\) as in equation \eqref{ProofThmSmallerClassControlsCLAIM} and compute
		\begin{align*}
			|J(t,\xi,\alpha)-J(t,\xi,\alpha')| &\leq |J(t,\xi,\alpha)-J(u,\xi,\alpha)| + |J(u,\xi,\alpha)-J(u,\xi,\alpha')| \\
			&+ |J(u,\xi,\alpha')-J(t,\xi,\alpha')| < 2 \epsilon \quad \text{(by \eqref{ProofThmSmallerClassControls2} and \eqref{ProofThmSmallerClassControlsCLAIM}).} 
		\end{align*}
		Therefore, for any \(\alpha \in \Adm\) there exists a control \(\alpha' \in \Adm_t\) such that
		\begin{align*}
			J(t, \xi, \alpha) &\leq 2 \epsilon + J(t, \xi, \alpha') 
			\leq 2 \epsilon + V_t(t,\xi).
		\end{align*}
		Then, taking the supremum over the set of feasible controls \(\Adm\) and	by the arbitrariness of \(\epsilon>0\), we conclude that
		\[ V(t, \xi) \leq V_t(t,\xi).  \]
		It is thus sufficient to prove that \eqref{ProofThmSmallerClassControlsCLAIM} holds true.
		The proof is based on the idea that, with a time \(u\) ``close enough'' to \(t\), it is possible to map any trajectory on the time interval \([0,u]\) into a trajectory on the ``small interval'' \([t,u]\); as a consequence, the additional source of randomness given by \(B_{[0,u]}\) can be reproduced by means of \((B_s^{(t)})_{s \in [t,u]}\) and does not improve the value obtained optimizing over \(\Adm\).
		
		We introduce the following ``scaling'' function, which maps every continuous path defined on \([t,u]\) to a continuous path defined on \([0,u]\):
		\begin{align}
			\phi: \Cont([t,u]) &\rightarrow \Cont([0,u]) \nonumber \\
			v=(v(s))_{s \in [t,u]} &\rightarrow \phi(v)=w=(w(s))_{s \in [0,u]}: \nonumber \\
			& w(s)= v\bigg( t+s \frac{u-t}{u} \bigg) \sqrt{\frac{u}{u-t}} \quad \text{for every } s \in [0,u].
		\end{align}
		For any \(\tilde{B}\) standard Brownian motion on \([t,u]\), we can easily verify that \(\tilde{B}'\) such that \(\tilde{B}'(\omega)=\phi(\tilde{B}(\omega))\) is a standard Brownian motion on \([0,u]\). Therefore, the push-forward of \(W_{[t,u]}\) through \(\phi\) corresponds to \(W_{[0,u]}\).
		Now consider equation \eqref{DisintegrationEquation} with initial condition \(\xi\) at time \(u\), i.e.
		\begin{align}
			\label{DisintegrationEquationfromu}
			dX_s^w&= b \bigg(X_s^w, \int_{\Cont([0,u])} \Prob_{X_s^{w'}} W_{[0,u]}(w'), \alpha_s^u(w)\bigg) ds \nonumber \\
			&+ \sigma \bigg(X_s^w, \int_{\Cont([0,u])} \Prob_{X_s^{w'}} W_{[0,u]}(w'), \alpha_s^u(w) \bigg) dB_s, \quad X_u^w=\xi
		\end{align}
		for every \(s \geq u\) and \(w \in \Cont([0,u])\). For every \(w \in \Cont([0,u])\), consider \(v \in \Cont([t,u])\) such that \(\phi(v)=w\); such \(v\) exists because \(\phi\) is invertible. We introduce the notation
		\begin{equation}
			\label{DefinitionofXv}
			X^v:=X^{\phi(v)}
		\end{equation}
		to denote the solution to equation \eqref{DisintegrationEquationfromu} labelled by \(\phi(v)\) and \begin{equation}
			\gamma_s^u(v):= \alpha_s^u(\phi(v)), \quad s \geq u
			\label{Definitionofgammauv}
		\end{equation}
		for the associated control. By the previous observation about the push-forward of the Wiener measure, we can see that
		\begin{equation}
			\label{PushForwardWienerMeasure}
			\int_{\Cont([0,u])} \Prob_{X_s^{w'}} W_{[0,u]}(dw')= \int_{\Cont([t,u])} \Prob_{X_s^{v'}} W_{[t,u]}(dv').
		\end{equation}
		We deduce that \( (X^v)_{v \in \Cont([t,u])} \) solves the system of equations
		\begin{align}
			\label{DisintegrationEquationfromuRewritten}
			dX_s^v&= b \bigg(X_s^v, \int_{\Cont([t,u])} \Prob_{X_s^{v'}} W_{[t,u]}(v'), \gamma_s^u(v)\bigg) ds \nonumber \\
			&+ \sigma \bigg(X_s^v, \int_{\Cont([t,u])} \Prob_{X_s^{v'}} W_{[t,u]}(v'), \gamma_s^u(v) \bigg) dB_s, \quad X_u^v=\xi
		\end{align}
		for every \(s \geq u\) and for every \(v \in \Cont([t,u])\). We now compose \(X^v\) with \(B_{[t,u]}:=(B_s^{(t)})_{t \leq s \leq u}\), which is independent of \(X^v\). Call
		\( \bar{X}:=X^v_{|v=B_{[t,u]}}  \)
		and 
		\( \bar{\gamma}:= \gamma^u(B_{[t,u]}); \)
		by the definition \eqref{Definitionofgammauv} of \(\gamma^u\), we have
		\(  \bar{\gamma}_s= \underline{\gamma}(s, \phi(B_{[t,u]}) \otimes_u B^{(u)}, U) \)
		for a suitable measurable function \(\underline{\gamma}\). As a consequence, our control \(\bar{\gamma} \in \Adm_t\) because it is a measurable function of \(B_{[t,u]}\), \(B^{(u)}\) and \(U\). We also notice that, repeating the same reasoning as in the proof of Lemma \ref{LemmaDisintegrationandCompositionBgivesX},
		\begin{equation}
			\label{DisintegrationofLawofbarX}
			\int_{\Cont([t,u])} \Prob_{X_s^v} W_{[t,u]}(dv) = \Prob_{\bar{X}_s},
		\end{equation}
		i.e.\ \(\bar{X}\) is the unique solution to equation
		\begin{equation}
			\label{EquationofbarX}
			d\bar{X}_s= b \bigg(\bar{X}_s, \Prob_{\bar{X}_s}, \bar{\gamma}_s \bigg) ds + \sigma \bigg(\bar{X}_s, \Prob_{\bar{X}_s}, \bar{\gamma}_s \bigg) dB_s, \quad \bar{X}_u=\xi,
		\end{equation}
		which is a McKean-Vlasov controlled equation of the form of equation \eqref{StateEquationFromr} with initial time \(u\); by the property of uniqueness of solutions to McKean-Vlasov SDEs, \(\bar{X}=X^{u, \xi, \bar{\gamma}}\).
		We now compute the gain associated with this dynamics: recalling \eqref{DefinitionofXv}, \eqref{Definitionofgammauv}, \eqref{PushForwardWienerMeasure} and the definition of push-forward, by conditional independence and \eqref{DisintegrationofLawofbarX},
		\begin{align*}
			&J(u, \xi, \bar{\gamma})\\
			&= \int_{\Cont([t,u])} \ExpVal \bigg[ \int_{u}^{+ \infty} e^{-\beta (s-u)} f( X_s^v, \int_{\Cont([t,u])} \Prob_{X_s^v} W_{[t,u]}(dv), \gamma_s^u(v)) ds \bigg] W_{[t,u]}(dv)  \\
			&= \int_{\Cont([0,u])} \ExpVal \bigg[ \int_{u}^{+ \infty} e^{-\beta (s-u)} f(X_s^{w}, \int_{\Cont([0,u])} \Prob_{X_s^{w'}} W_{[0,u]}(dw'), \alpha_s^u(w) ) ds \bigg] W_{[0,u]}(dw) \\
			&= \ExpVal \bigg[ \int_{u}^{+ \infty} e^{-\beta (s-u)} f(X_s^{u, \xi, \alpha}, \Prob_{X_s^{u,\xi,\alpha}}, \alpha_s) ds \bigg]= J(u,\xi, \alpha),
		\end{align*}
		where in the second-to-last equality we applied Lemma \ref{LemmaDisintegrationandCompositionBgivesX} and equation \eqref{DisintegrationEquationfromu}.
	\end{proof}

	\begin{Osservazione}
		The basic arguments in the 
		previous proof as well as in Example \ref{CounterExample}
		are based on a randomization procedure, where the source of noise was derived from the trajectory of the Brownian motion  preceding the initial time.  
		One could randomize in a similar way by means of the independent random variable  \(U\). The previous proof has the advantage that it applies to situations where the initial condition is deterministic and the \(\sigma\)-algebra
		\(\GG\) can be trivial. \qed
	\end{Osservazione}
	By Proposition \ref{VaryinginTimeTimeInvarianceProperty} and Theorem \ref{ThmValueOverSmallerClassControls} we immediately conclude the following.
	\begin{Teorema}[Time invariance property]
		\label{TimeInvarianceProperty}
		Suppose that assumptions (\nameref{HpCoefficients}), (\nameref{HpGain}) (or (\nameref{HpGainAlternative})) and (\nameref{HpContf}) are satisfied. Let \(\xi \in L^2(\Omega, \GG, \Prob).\) Then,
		\begin{equation}
			\label{EquationTimeInvariance}
			V(0, \xi)=V(t,\xi) \qquad \forall t \geq 0,
		\end{equation}
		i.e.\ the value associated to a \(\GG\)-measurable initial condition does not depend on the initial time.
	\end{Teorema}
	
	\section{Continuity of the Value Function and Law Invariance Property}
	\label{SectionPropertiesVF}
	In this section, we discuss the properties of the lifted value function.
	\begin{Proposizione}\label{ThmGrowthConditionsofV}
		Under assumptions (\nameref{HpCoefficients}) and (\nameref{HpGain}), for any \(t \geq 0\) the lifted value function \(V\) at time \(t\) satisfies a quadratic growth condition in \(\xi\), i.e.\ there exists a positive constant \(C\) such that
		\begin{equation}
			\label{QuadraticGrowthofV}
			|V(t,\xi)| \leq C (1+||\xi||_{L^2}^2)
		\end{equation}
		for every \(\xi \in L^2(\Omega,\FF_t,\Prob)\).
		On the other hand, under assumptions (\nameref{HpCoefficients}) and (\nameref{HpGainAlternative}), the lifted value function \(V\) is bounded, i.e.\ there exists a constant \(C >0\) such that \(|V(t,\xi)| \leq C\)
		for every \(t \geq 0\) and \(\xi \in L^2(\Omega,\FF_t,\Prob)\).
	\end{Proposizione}
	\begin{proof}
		Firstly, notice that
		\(|V(t,\xi)| \leq \sup_{\alpha \in \Adm} |J(t,\xi,\alpha)|;\)
		it is thus sufficient to prove the desired bounds for the functional \(J(t,\cdot,\alpha)\) uniformly with respect to \(\alpha\). If assumptions (\nameref{HpCoefficients}) and (\nameref{HpGain}) are satisfied,
		\begin{align*}
			|J(t,\xi,\alpha)|
			& \leq \ExpVal \bigg[ \int_{t}^{+ \infty} K e^{-\beta (s-t)} (1+|X_s^{t,\xi,\alpha}|^2+||\Prob_{X_s^{t,\xi,\alpha}}||_2^2) ds \bigg] \quad \text{(by \eqref{InitialConditionEstimate})} \\
			& \leq \frac{K}{\beta} + \int_{t}^{+ \infty} 2K e^{(-\beta+12M^2 +1) (s-t)} \bigg(|| \xi ||_{L^2}^2 + 6M^2 (s-t) \bigg) ds \\
			& \leq C(1 + || \xi ||_{L^2}^2)
		\end{align*}
		where \(C=C(K,M,\beta)\) is a positive and finite constant independent of \(\alpha\).
		
		On the other hand, if Assumptions (\nameref{HpCoefficients}) and (\nameref{HpGainAlternative}) hold, the thesis easily follows from the boundedness of the function \(f\).
	\end{proof}
	
	From now on, we always assume that (\nameref{HpCoefficients}), (\nameref{HpGain}) (or (\nameref{HpGainAlternative})) and (\nameref{HpContf}) are satisfied.
	\begin{Proposizione} \label{ThmContinuityoftheValueFunction}
		The lifted value function \(V\) is right continuous with respect to time, i.e.\, for any \(t \geq 0\) and \(\xi \in L^2(\Omega, \FF_t, \Prob)\),  \(V(u, \xi) \rightarrow V(t,\xi)\) when \(u \rightarrow t^+\).
		Moreover, \(V\) is continuous with respect to the initial condition, i.e.\: for every  \(t \geq 0\), \(\xi \in L^2(\Omega, \FF_t, \Prob)\), \(V(t, \eta) \rightarrow V(t,\xi)\) when \(\eta \rightarrow \xi\) in \(L^2(\Omega, \FF_t, \Prob)\).
	\end{Proposizione} 
	\begin{proof}
		The right continuity of \(V\) easily follows from Theorem \ref{ThmContinuityofJwrtTime}.
		The continuity with respect to the initial condition can be obtained following the lines of the proof of Theorem \ref{ThmContinuityofJwrtTime} in \ref{AppendixEquaivalenctFormulation}, using the following bound, derived from standard estimates on a bounded interval: for any \(T>0\) there exists a constant \(C(T)>0\) such that, for any initial time \(r \in [u,T]\), any two initial conditions, \(\xi, \eta \in L^2(\Omega, \FF_r, \Prob)\), and any feasible control \(\alpha \in \Adm\),
		\(\ExpVal [ |X_s^{r,\xi,\alpha}-X_s^{r,\eta,\alpha}|^2 ] \leq 3 C(T) \ExpVal [ |\xi-\eta|^2 ]\) for all \(s \in [r,T]\).
	\end{proof}
	
	In analogy with the finite horizon case, presented in \cite{CossoExistence}, we want to redefine the \textit{value function} of the problem as a real-valued function on \(\PPtwo\). This requires the so-called Law Invariance Property, stating that the lifted value function \(V\) depends on the initial condition only through its probability distribution (which motivates the choice of the expression ``lifted value function'').
	\begin{Teorema}[Law Invariance Property]
		\label{LawInvarianceTheorem}
		Fix any \(t \geq 0\) and any  \(\xi, \eta \in L^2(\Omega, \FF_t, \Prob)\) such that \(\Prob_{\eta}=\Prob_{\xi}\). Then, \(V(t,\eta)=V(t,\xi)\)
	\end{Teorema}
	\begin{proof}
		The result can be proved following the lines of the proof of Theorem 3.6 in \cite{CossoExistence} with minor changes.
	\end{proof}
	
	Therefore, thanks to Theorem \ref{TimeInvarianceProperty} and Theorem \ref{LawInvarianceTheorem}, we can define the \textit{value function} \(v: \PPtwo \rightarrow \R\) as 
	\[ v(\mu)=V(0, \xi) \quad \text{for any } \mu \in \PPtwo  \]
	for every \(\xi \in L^2(\Omega, \GG, \Prob)\) such that \(\Prob_{\xi}=\mu\); furthermore, we can easily verify that \(v(\mu)=V(t,\eta)\) for any \(t \geq 0\) and \(\eta \in L^2(\Omega, \FF_t, \Prob)\) such that \(\Prob_{\eta}=\mu\).
	Consequently, the main properties of \(v\) can be deduced from the properties of the lifted value function. We start by stating the dynamic programming principle for the value function \(v\), which follows from Theorem \ref{DPPThm} with \(t=0\).
	\begin{Corollario}[DPP] \label{DPPThmv}
		For any \(s \geq 0\) and for every \(\mu \in \PPtwo\),
		\begin{equation} 
			\label{DPPvEquation}
			v(\mu) = \sup \limits_{\alpha \in \Adm} \biggl\{ \ExpVal \bigg[ \int_{0}^{s} e^{-\beta r} f(X_r^{0, \xi, \alpha}, \Prob_{X_r^{0, \xi, \alpha}}, \alpha_r) dr \bigg]  + e^{-\beta s} v(\Prob_{X_{s}^{0, \xi, \alpha}}) \biggr\}
		\end{equation}
		for any \(  \xi \in L^2(\Omega, \GG, \Prob) \) such that \(\Prob_{\xi}=\mu\).
	\end{Corollario}
	
	We then deduce the following result from Proposition \ref{ThmGrowthConditionsofV}.
	\begin{Corollario} \label{ThmGrowthConditionsofv}
		Under assumptions (\nameref{HpGain}), the value function \(v\) satisfies a quadratic growth condition in \(\mu\), i.e.\ there exists a positive constant \(C\) such that	\(|v(\mu)| \leq C (1+||\mu||_2^2)\)
		for every \(\mu \in \PPtwo\).
		On the other hand, if (\nameref{HpGainAlternative}) hold, the value function \(v\) is bounded.
	\end{Corollario}
	
	We conclude the section studying the continuity of the value function \(v\), which will be a key result for its the viscosity property.
	\begin{Proposizione}
		\label{PropContinuityv}
		The value function \(v\) is continuous, i.e.\ for every \(\mu \in \PPtwo\) and any sequence \((\mu_n)_{n \in \N} \subseteq \PPtwo\) converging to \(\mu\) it holds that \(\lim \limits_{n \to +\infty} v(\mu_n) = v(\mu)\).
	\end{Proposizione}
	\begin{proof}
		The thesis follows from Theorem \ref{ThmContinuityoftheValueFunction}, as, by Lemma 2.1 in \cite{BandiniCossoFuhrmanPhamRandomizedFiltering}, the convergence with respect to the 2-Wasserstein metric of the sequence \(\mu_n\) to \(\mu\) implies the convergence in \(L^2(\Omega, \GG, \Prob)\) of a sequence \((\xi_n)_{n \in \N} \subseteq L^2(\Omega, \GG, \Prob)\) with \(\Prob_{\xi_n}=\mu_n\) for any \(n \in \N\) to some \(\xi \in L^2(\Omega, \GG, \Prob)\) with \(\Prob_{\xi}=\mu\).
	\end{proof}
	
	\section{Hamilton-Jacobi-Bellman equation}
	\label{SectionHJB}
	\subsection{Lions Derivative and viscosity solutions to PDEs}
	\label{SectionDerivativesandViscosityTheorey}
	In this section, we are studying our optimal control problem by means of the classical approach based on Hamilton-Jacobi-Bellman equation: we aim at deriving from the Dynamic Programming Principle, Theorem \ref{DPPThmv}, a partial differential equation on the domain of the value function, in order to characterize \(v\) as its unique solution. As in many applications of control theory it is not easy or even possible to prove the regularity of the value function, we expect it to solve a PDE in the \textit{viscosity sense}.
	We first need to introduce an appropriate notion of derivative on the space \(\PPtwo\); we are using the derivative in the lifted sense of Lions (\cite{LionsCours}, see also \cite{Cardaliaguet2013}). We recall here that, given a map \(u:\PPtwo \rightarrow \R\), we call \textit{lifting of \(u\)} any function \(U:L^2(\Omega, \FF, \Prob; \R^d) \rightarrow \R\) such that \(U(\xi)=u(\Prob_{\xi})\) for any \(\xi \in L^2(\Omega, \FF, \Prob; \R^d)\).
	\begin{Definizione}
		A function \(u: \PPtwo \rightarrow \R\) is said to be first-order L-differentiable if its lifting \(U\) admits a continuous Fréchet derivative \(D_{\xi}U:L^2(\Omega, \FF, \Prob; \R^d) \rightarrow L^2(\Omega, \FF, \Prob; \R^d)\).
	\end{Definizione}
	
	\begin{Osservazione}
		By Theorem 6.5 in  \cite{Cardaliaguet2013}, if \(u\) is first-order differentiable, then there exists, for any \(\mu \in \PPtwo\), a measurable function \(\partial_{\mu}u(\mu): \R^d \rightarrow \R^d \) such that \(\Prob\)-a.s.\ \(D_{\xi}U(\xi)= \partial_{\mu}u(\mu)(\xi) \) for any \(\xi \in L^2(\Omega, \FF, \Prob; \R^d)\) with \(\Prob_{\xi}=\mu\). The function \(\partial_{\mu}u(\mu)\) (defined \(\mu\)-a.s.\ ) is called first-order Lions derivative of \(u\) at \(\mu\). \qed
	\end{Osservazione}
	
	\begin{Definizione}
		We denote by \(\Cont^{2}(\PPtwo)\) the space of continuous functions \(\phi: \PPtwo \rightarrow \R\) such that the following conditions hold: 
		\begin{enumerate}
			\item \(\phi\) is first-order L-differentiable;
			\item The map \((\mu, x) \in \PPtwo \times \R^d \rightarrow \partial_{\mu} \phi(\mu)(x) \in \R^d\) is jointly continuous.
			\item The derivative \(\partial_x \partial_{\mu} \phi: (\mu,x) \in \PPtwo \times \R^d \rightarrow \partial_x \partial_{\mu} \phi (\mu)(x) \in \R^{d \times d}\) exists and is jointly continuous.
		\end{enumerate}
	\end{Definizione} 
	\begin{Osservazione}
		This is not the only available notion of derivative with respect to a measure; we choose the derivative in the lifted sense because we defined the function \(v\) starting exactly from its lifting, the function \(V\). Moreover, this choice is consistent with a part of the existing literature (see e.g.\ \cite{CossoExistence}, \cite{CossoUniqueness}, \cite{PhamWei1}, \cite{PhamWei2}) from which this work was originally inspired. We refer to Section 5 in \cite{CarmonaDelarue1} for the definition of other notions of differentiability, which are equivalent to our definition under suitable regularity conditions (see e.g.\ Proposition 5.51 and Proposition 5.48 in \cite{CarmonaDelarue1}). \qed
	\end{Osservazione}
	
	We now introduce the set of test functions used in this paper.
	\begin{Definizione}
		We denote by \(\TestFcts(\R^d)\) the subset of \(C^2(\PPtwo)\) of functions \(\phi :\PPtwo \rightarrow \R\) such that, for some constant \(C_{\phi} \geq 0\),
		\begin{equation*}
			|\partial_{\mu} \phi(\mu)(x)| \leq C_{\phi} (1+|x|)
			\text{ and }|\partial_x \partial_{\mu} \phi (\mu)(x)| \leq C_{\phi}
		\end{equation*}
		for every \(\mu \in \PPtwo\) and \(x \in \R^d\).
	\end{Definizione}

	On the space \(\PPtwo\) endowed with the \(2\)-Wasserstein space, we introduce the following elliptic partial differential equation, which we will refer to as Hamilton-Jacobi-Bellman (HJB) equation:
	\begin{align}
		\label{HJB}
		\beta v(\mu) &= \ExpVal \bigg[ \sup \limits_{a \in A} \biggl\{ f(\xi, \mu, a ) + <b(\xi, \mu,a ), \partial_{\mu} v(\mu)(\xi)> \nonumber \\
		& + \frac{1}{2} tr \bigg(\sigma(\xi, \mu,a ) \sigma^T (\xi, \mu,a ) \partial_{x}\partial_{\mu}v(\mu)(\xi) \bigg) \biggr\} \bigg]
	\end{align}
	for \(\mu \in \PPtwo\) and any \(\xi \in L^2(\Omega, \GG, \Prob)\) with \(\Prob_{\xi}=\mu\).
	
	We want to define a suitable notion of viscosity solution to HJB equation.
	\begin{Definizione}
		A continuous function \(u:\PPtwo \rightarrow \R\) is a viscosity subsolution to equation \eqref{HJB} if, for every \(\mu \in \PPtwo\) and every \(\phi \in \TestFcts(\R^d)\) such that \(u- \phi \) attains a global maximum at \(\mu\) with value \(0\), the following inequality is satisfied:
		\begin{align}
			\label{SubsolHJB}
			0 &\leq - \beta \phi(\mu) +\ExpVal \bigg[ \sup \limits_{a \in A} \biggl\{ f(\xi, \mu, a ) + <b(\xi, \mu,a ), \partial_{\mu} \phi(\mu)(\xi)> \nonumber \\
			& + \frac{1}{2} tr \bigg(\sigma(\xi, \mu,a ) \sigma^T (\xi, \mu,a ) \partial_{x}\partial_{\mu}\phi(\mu)(\xi) \bigg) \biggr\} \bigg]
		\end{align}
		for every \(\xi \in L^2(\Omega, \GG, \Prob)\) with law \(\mu\).
	\end{Definizione}
	
	In order to define supersolutions, we underline that any function \(u:\PPtwo \rightarrow \R\) can be extended to the space \(\PPtwoA\) in a canonical way: for any \(\nu \in \PPtwoA, \) \(u(\nu):=u(\nu_1)\) where \(\nu_1\) is the marginal of \(\nu\) such that \(\nu(\cdot \times A)=\nu_1(\cdot)\). We also introduce the projection on the first \(d\) components \(\pi_d:\R^d \times A \rightarrow \R^d\) where, for any \(y=(y_1, \ldots, y_d, y_{d+1}, \ldots) \in \R^d \times A\), \(\pi_d(y)=(y_1, \ldots, y_d)\).
	\begin{Definizione} \label{DefinitionSupersol}
		A continuous function \(u:\PPtwo \rightarrow \R\) is a viscosity supersolution to equation \eqref{HJB} if, for any \(\nu \in \PPtwoA\) and for any \(\psi \in \TestFcts(\R^d \times A)\) such that \(u-\psi\) attains a minimum with value \(0\) at \(\nu\) over the space \(\PPtwoA\), the following inequality holds:
		\begin{align}
			\label{SupersolHJB}
			0 &\geq - \beta \psi(\nu) +\int_{\R^d \times A} \biggl\{ f(x, \mu, a ) + <b(x, \mu,a ), \partial_{\mu} \psi(\nu)(x,a)> \nonumber \\
			& + \frac{1}{2} tr \bigg(\sigma(x, \mu,a ) \sigma^T (x, \mu,a ) \partial_{x}\partial_{\mu}\psi(\nu)(x,a) \bigg) \biggr\} \nu(dx,da),
		\end{align}
		where \(\mu\) is the marginal distribution of \(\nu\) on \(\R^d\) and \(\partial_{\mu}\psi(\nu)(\cdot,\cdot):= \pi_d(\partial_{\nu}\psi(\nu)(\cdot,\cdot))\).
	\end{Definizione}
	
	\begin{Definizione}
		We call \(u\) a viscosity solution to equation \eqref{HJB} if it is both a supersolution and a subsolution to it.
	\end{Definizione}
	
	\begin{Osservazione}
		\label{RemarkEquivalentFormulationSupersolutions}
		A more natural definition of supersolutions can be introduced: indeed, we may define supersolutions ``à la Crandall-Lions'' to equation \eqref{HJB} as continuous functions \(u: \PPtwo \rightarrow \R\) such that, 
		for every \(\mu \in \PPtwo\) and every \(\phi \in \TestFcts(\R^d)\) such that \(u- \phi \) attains a global minimum at \(\mu\) with value \(0\), inequality \eqref{SubsolHJB} is satisfied with sign \(\geq\) instead of \(\leq\).
		As pointed out in \cite{CheungTaiQiu}, Remark 6.1, and in \cite{BayraktarCheungTaiQiu}, Remark 2.5, this definition of viscosity supersolution is not sufficient to prove uniqueness of viscosity solutions to Hamilton-Jacobi-Bellman equation. We thus prefer to define supersolutions as in Definition \ref{DefinitionSupersol}, adapting the notion presented in \cite{CheungTaiQiu} and \cite{BayraktarCheungTaiQiu} to an elliptic equation. 
		Nonetheless, under Assumptions (\nameref{HpCoefficients}), (\nameref{HpGain}) or (\nameref{HpGainAlternative}) and (\nameref{HpContf}), inequality \eqref{SupersolHJB} implies that \(u\) is also a supersolution ``à la Crandall-Lions'' (see Remark 6.1 in \cite{CheungTaiQiu})). We furthermore underline that we could prove that \(v\) is a supersolution to equation \eqref{HJB} even in the weaker Crandall-Lions sense. \qed	\end{Osservazione}
	
	\subsection{Hamilton-Jacobi-Bellman equation and value function}
	\label{SectionHJBExistence}
	We now prove that \(v\) solves the Hamilton-Jacobi-Bellman equation associated with our optimal control problem.
	\begin{Teorema}
		\label{ThmvViscositySolHJB}
		Under hypotheses (\nameref{HpCoefficients}), (\nameref{HpGainAlternative}) and (\nameref{HpContf}),
		the value function \(v\) is a viscosity solution to Hamilton Jacobi Bellman equation \eqref{HJB}.
	\end{Teorema}
	\begin{proof}
		The result follows from an adaptation of Theorem 6.2 in \cite{CheungTaiQiu} and Theorem 4.2 in \cite{BayraktarCheungTaiQiu}, with slight modifications due to the elliptic structure of the equation.
	\end{proof}
	\begin{Osservazione}
		The result is still valid if we substitute assumptions (\nameref{HpGainAlternative}) with (\nameref{HpGain}), with minor changes in the proof; it is sufficient to follow the lines of Theorem 5.3 in \cite{CossoExistence}. Notice that the boundedness of the test functions is not required in our result. \qed
	\end{Osservazione}

	\subsection{Uniqueness of the solution}
	\label{SectionUniquenessHJB}
	We now discuss uniqueness of solutions to equation \eqref{HJB}. To the best of our knowledge, a comparison principle for viscosity solutions to elliptic partial differential equations on the space \(\PPtwo\) is not directly available. We thus proceed by an approximation argument: starting from a sub/super solution \(u\) to the elliptic equation \eqref{HJB}, for each time \(T>0\) we construct a suitable finite horizon problem on \([0,T]\) with associated value function \(v_T\), solving (in the viscosity sense) a parabolic Hamilton-Jacobi-Bellman equation on \([0,T] \times \PPtwo\), as stated e.g.\ in \cite{CossoExistence}, and a sub/super solution to the same parabolic equation \(U_T\) such that \(U_T(0,\cdot)=u(\cdot)\). For the class of Hamilton-Jacobi-Bellman parabolic PDEs, there exist comparison results, ensuring that \(v_T \geq U_T\) (respectively, \(\leq U_T\)) for \(U_T\) sub/super solution to the parabolic HJB equation.
	Passing to the limit \(T \rightarrow +\infty\), we show that \(v_T(0,\cdot)\) converges to our value function \(v\) and conclude that \eqref{HJB} also admits a unique viscosity solution, corresponding to \(v\).
	The proof of our uniqueness result will thus heavily rely on a comparison principle for parabolic PDEs on \(\PPtwo\); therefore, in this section we will need to impose stronger assumptions, guaranteeing that the comparison principle holds for all the finite horizon approximations of our infinite horizon problem.
	Following \cite{BayraktarCheungTaiQiu}, assume that:
	\begin{Ipotesi}[\(\boldsymbol{H_{un}}\)] \label{HpUniqueness}
		\begin{itemize}
			\item[] 
			\item[(\(H_{un}(1)\))] The Polish space \(A\) is a compact subset of an Euclidean space;
			\item[(\(H_{un}(2)\))] The function \(\sigma\) is independent on the measure argument, i.e.\ \(\sigma:\R^d \times A \rightarrow \R^{d \times m}\);
			\item[(\(H_{un}(3)\))] The functions \(b\), \(\sigma\), \(f\) are jointly continuous in all their arguments (with respect to the \(W_1\) distance in the measure argument);
			\item[(\(H_{un}(4)\))] The functions \(b\), \(\sigma\) and \(f\) are Lipschitz in \((x,\mu)\) uniformly with respect to \(a\) in the \(W_1\) Wasserstein distance, i.e.\ there exists a constant \(L' >0\) such that
			\begin{align*}
				&|b(x,\mu,a)-b(x',\mu',a)|+|\sigma(x,a)-\sigma(x',a)|\\
				+&|f(x,\mu,a)-f(x',\mu',a)| \leq L'(|x-x'|+W_1(\mu,\mu'))
			\end{align*}
			for all \(x,x' \in \R^d\), \(\mu,\mu' \in \PPtwo\) and \(a \in A\).
			\item[(\(H_{un}(5)\))] The functions \(b\), \(\sigma\) and \(f\) are bounded, i.e.\ there exists \(C'>0\) such that
			\[ |b(x,\mu,a)|+|\sigma(x,a)|+|f(x,\mu,a)| \leq C'  \]
			for all \(x \in \R^d\), \(\mu \in \PPtwo\) and \(a \in A\);
			\item[(\(H_{un}(6)\))] For any fixed \(a \in A\), the function \(\sigma(\cdot,a)\) belongs to \(C^2(\R^d)\); moreover, there exists a constant \(\Sigma>0\) such that
			\[ |\partial_{x_i} \sigma(x,a)|+|\partial_{x_i x_j}^2 \sigma(x,a)| \leq \Sigma  \]
			for all \((x,a) \in \R^d \times A\) and all \(i=1, \ldots,d\), \(j=1, \ldots, m\);
			\item[(\(H_{uc}(7)\))] The discount rate \(\beta>0\) is sufficiently large.
		\end{itemize}
	\end{Ipotesi}
	
	\begin{Osservazione}
		Notice that the \(W_1\)-Lipschitz continuity of all the coefficients in Assumption (\(H_{un}(4)\)) is the main novelty introduced by \cite{CheungTaiQiu} in order to correct the uniqueness result presented in \cite{CossoUniqueness}, as explained in Remark 5.1 \cite{CheungTaiQiu}. The lower bound for the admissible value of \(\beta\) will be made explicit in the proof of the following Lemma. \qed
	\end{Osservazione}

	\begin{Lemma}
		\label{LemmavW1Lipschitz}
		Under assumptions (\nameref{HpUniqueness}), the value function \(v\) is \(W_1\)-Lipschitz, i.e.\ there exists a constant \(\tilde{L}>0\) such that, for any \(\mu,\nu \in \PPtwo\),
		\[ |v(\mu)-v(\nu)| \leq \tilde{L} W_1(\mu,\nu).  \]
	\end{Lemma}
	\begin{proof}
		The proof is based on the following claim: there exists a constant \(\tilde{L} \geq 0\) such that, for any \(\xi, \eta \in L^2(\Omega, \GG, \Prob)\),
		\[ |J(0,\xi,\alpha)- J(0,\eta,\alpha)| \leq \tilde{L} ||\xi - \eta||_{L^1} \text{ for any } \alpha \in \Adm.\]
		As a matter of fact, from this inequality we easily deduce that
		for any \(\mu, \nu \in \PPtwo\),
		\begin{align*}
			|v(\mu)-v(\nu)| &= \inf \{ |V(0,\xi)-V(0,\eta)|: \xi, \eta \in L^2(\Omega, \GG, \Prob) \text{ with } \Prob_{\xi}=\mu, \Prob_{\eta}=\nu \} \\
			&\leq \inf \{ \tilde{L} ||\xi-\eta||_{L^1}: \xi, \eta \in L^2(\Omega, \GG, \Prob) \text{ with } \Prob_{\xi}=\mu, \Prob_{\eta}=\nu \} \\
			&= \tilde{L} W_1(\mu,\nu)
		\end{align*}
		as \(V(0,\cdot)\) is law invariant and \(\GG\) is a sufficiently rich \(\sigma\)-algebra, and the thesis holds.
		
		In order to prove the claim, fix any two initial conditions \(\xi, \eta \in L^2(\Omega, \GG, \Prob)\) and denote the solution to the state equation \eqref{StateEquation} starting at time \(t=0\) from \(\xi\) and \(\eta\), respectively, by \(X^{\xi}\) and \(X^{\eta}\). By standard estimates relying on Burkholder-Davis-Gundy inequality (see e.g.\ Theorem 5.16 in \cite{LeGall}) and (\nameref{HpUniqueness}),
		\begin{equation}
			\label{ProofvLipschitzIteration1}
			\ExpVal \bigg[ \sup \limits_{0 \leq r \leq s} |X_r^{\xi}-X_r^{\eta}| \bigg] \leq 2  ||\xi-\eta||_{L^1}.
		\end{equation}
		for \(0 \leq s \leq s_0:=s_0(L',C_1)\), where 
		\(C_1\) is the constant appearing in Burkholder-Davis-Gundy inequality for \(p=1\). Iterating \eqref{ProofvLipschitzIteration1}, we deduce that
		\begin{equation}
			\label{ProofvLipschitzIteration3}
			\ExpVal \bigg[ \sup \limits_{0 \leq r \leq ks_0} |X_r^{\xi}-X_r^{\eta}| \bigg] \leq 2^k ||\xi-\eta||_{L^1}.
		\end{equation}
		Now, for fixed \(s \geq 0\), we can always find an integer \(k\) such that \((k-1)s_0 \leq s \leq ks_0\). As a consequence,
		\begin{equation}
			\label{ProofvLipschitzInitialConditionsEstimate}
			\ExpVal \bigg[ \sup \limits_{0 \leq r \leq s} |X_r^{\xi}-X_r^{\eta}| \bigg] \leq e^{\bar{C_0}(s+1)} ||\xi-\eta||_{L^1}
		\end{equation}
		with \(\bar{C_0}=\bar{C_0}(s_0)  \). 
		Therefore, by (\nameref{HpUniqueness}),
		\begin{align*}
			|J(0,\xi,\alpha)- J(0,\eta,\alpha)|
			&\leq \ExpVal \bigg[ \int_{0}^{+ \infty} e^{-\beta s} L' \bigg(|X_s^{\xi}-X_s^{\eta}| + W_1(\Prob_{X_s^{\xi}}, \Prob_{X_s^{\eta}})  \bigg) ds \bigg] \quad \text{(by \eqref{ProofvLipschitzInitialConditionsEstimate})} \nonumber \\
			&\leq 2L'e^{\bar{C_0}} \int_{0}^{+ \infty} e^{(-\beta + \bar{C_0}) s} ds ||\xi-\eta||_{L^1} \quad \text{(if } \beta > \bar{C_0}) \nonumber \\
			&=\tilde{L} ||\xi-\eta||_{L^1},
		\end{align*}
		where \(\tilde{L}\) is a finite constant independent of \(\alpha\), provided that \(\beta > \bar{C_0}\). 
	\end{proof}
	
	We now state and prove the main result of this section, which characterizes the value function \(v\) as the unique viscosity solution to equation \eqref{HJB} in a specific class of functions.
	\begin{Teorema}
		\label{ThmUniqueness}
		Under assumptions (\nameref{HpUniqueness}), the value function \(v\) of our control problem is the unique bounded and \(W_1\)-Lipschitz viscosity solution to equation \eqref{HJB}.
	\end{Teorema}
	
	\begin{proof}
		Notice that under Assumptions (\nameref{HpUniqueness}) \(v\) is \(W_2\)-continuous, bounded and also \(W_1\)-Lipschitz, see Corollary \ref{ThmGrowthConditionsofv},  Proposition \ref{PropContinuityv}, and Lemma \ref{LemmavW1Lipschitz}. Furthermore, by Theorem \ref{ThmvViscositySolHJB}, it is sufficient to prove the following: for any bounded and \(W_1\)-Lipschitz viscosity subsolution \(u\) to Hamilton-Jacobi-Bellman equation \eqref{HJB}, it holds that
		\begin{equation}
			\label{vsmallersupersolution}
			v(\mu) \geq u(\mu) \quad \forall \mu \in \PPtwo
		\end{equation}
		and, for any bounded and \(W_1\)-Lipschitz viscosity supersolution \(w\)  to Hamilton-Jacobi-Bellman equation \eqref{HJB}, it holds that 
		\begin{equation}
			\label{vbiggersubsol}
			w(\mu) \geq v(\mu) \quad \forall \mu \in \PPtwo. 
		\end{equation}
		
		Fix any bounded and \(W_1\)-Lipschitz subsolution \(u\) to equation \eqref{HJB}. We first approximate \(v\) with a suitable sequence of value functions of finite horizon problems. 
		For any  \(T>0\), we aim at maximizing the gain functional
		\[  J(t,\xi,\alpha) := \ExpVal \bigg[ \int_{t}^{T} e^{-\beta s} f(X_s,\Prob_{X_s}, \alpha_s) ds \bigg] + e^{-\beta T} u(\Prob_{X_T})  \]
		over the feasible controls belonging to
		\( \Adm^T:= \{ \alpha: [0,T] \times \Omega \rightarrow A: \alpha \text{ is } \F \text{-predictable}  \}  \)
		under the dynamics
		\begin{equation}
			\label{StateEquationTruncatedT}
			\begin{cases}
				dX_s= b(X_s, \Prob_{X_s}, \alpha_s) ds + \sigma(X_s, \alpha_s) dB_s \quad s \in [t, T] &\\
				X_t= \xi&
			\end{cases}
		\end{equation} 
		for \(t \in [0,T]\) and \(\xi \in L^2(\Omega, \FF_t, \Prob)\).
		We define the lifted value function for our finite horizon control problem as
		\( V_T(t,\xi) = \sup_{\alpha \in \Adm^T} J(t,\xi,\alpha); \)
		by Theorem 3.6 in \cite{CossoExistence}, the function \(V_T\) is law invariant and we can thus define the true value function of the problem as
		\[  v_T(t,\mu) := V_T(t,\xi) \qquad \text{ for any } \xi \in L^2(\Omega,\FF_t,\Prob) \text{ having law } \mu. \]
		
		Notice that 
		\( \Adm^T= \{\alpha_{|_{[0,T]}}: \alpha \in \Adm\},  \) which implies
		\begin{align*}
			|v(\mu)-v_T(0,\mu)|
			& \leq \sup_{\alpha \in \Adm} \bigg\{  \ExpVal \bigg[ \int_{T}^{+ \infty}  e^{- \beta s} | f(X_s, \Prob_{X_s}, \alpha_s)| ds \bigg] + \bigg| e^{-\beta T} u(\Prob_{X_T}) \bigg| \bigg\} \\
			&\leq \tilde{C} \bigg(\frac{1}{\beta}+1 \bigg) e^{-\beta T} \xrightarrow{T \to +\infty} 0 \qquad \forall \mu \in \PPtwo,
		\end{align*}
		for some constant \(\tilde{C}\), where we used the boundedness of \(u\) and of \(f\).
		Thus, the sequence \(v_T(0,\cdot)\) approximates \(v\). 
		
		We now recall that, by Theorem 4.2 in \cite{BayraktarCheungTaiQiu}, \(v_T\) solves in the viscosity sense the following parabolic Hamilton-Jacobi-Bellman equation:
		\begin{equation}
			\label{ParabolicHJB}
			\begin{cases}
				\displaystyle
				-\partial_{t} w(t,\mu)= \int_{\R^d} \bigg[ \sup \limits_{a \in A} \bigg\{ e^{-\beta t} f(x,\mu,a) + <b(x,\mu,a), \partial_{\mu} w(t,\mu)(x)> \\ 
				+ \frac{1}{2} tr(\sigma(x,a) \sigma^T(x,a) \partial_{x} \partial_{\mu} w(t,\mu)(x)) \bigg\} \bigg] \mu(dx)   \quad (t,\mu) \in [0,T] \times \PPtwo&\\
				w(T,\mu)= e^{-\beta T} u(\mu) \quad \mu \in \PPtwo&
			\end{cases} 
		\end{equation}
		Moreover, as \(e^{-\beta t} f(x,\mu,a)\) and \(g(\mu):=e^{-\beta T} u(\mu)\) satisfy Assumptions \textbf{(A)} in \cite{BayraktarCheungTaiQiu}, by the comparison principle in the parabolic case, i.e.\ Theorem 4.3 in \cite{BayraktarCheungTaiQiu}, 
		\begin{equation}
			\label{ProofUniquenessSubsolution}
			v_T(t,\mu) \geq U_T(t,\mu) \quad \text{for any } t\in [0,T], \mu \in \PPtwo
		\end{equation}
		for any \(W_1\)-continuous and bounded subsolution \(U_T\) to equation \eqref{ParabolicHJB}.
		
		We now want to construct an appropriate subsolution to \eqref{ParabolicHJB}.	  
		Consider
		\( U_T(t,\mu):= e^{-\beta t} u(\mu)  \) for all \(t \in [0,T], \, \mu \in \PPtwo.\)
		Notice that \(U_T(0,\mu)=u(\mu)\) for any \(\mu \in \PPtwo\). Moreover, for any \(\mu \in \PPtwo\),  \(U_T(\cdot, \mu) \in C^{\infty}(0,T)\) with \(\partial_{t} U_T(t,\mu)=-\beta U_T(t,\mu)\). We prove that \(U_T\) is a viscosity subsolution to \eqref{ParabolicHJB} (in the sense of Definition 2.5 in \cite{BayraktarCheungTaiQiu}). 
		Firstly, \(U_T\) is bounded, \(([0,T], |\cdot|) \times (\PPtwo, W_1)\)- continuous and the boundary condition is verified as \(U_T(T,\mu)=e^{-\beta T}u(\mu)\) for any \(\mu \in \PPtwo\). Fix a function \(\psi\) such that the derivatives \(\partial_{t} \psi\), \(\partial_{\mu} \psi\), \(\partial_x \partial_{\mu} \psi\) exist and are jointly continuous and there exists a constant \(C_{\psi}\) such that
		\[ |\partial_{t} \psi(t,\mu)|+| \partial_x \partial_{\mu} \psi(t,\mu)(x)| \leq C_{\psi} \text{ and } |\partial_{\mu} \psi(t,\mu)(x)| \leq C_{\psi}(1+|x|) \]
		for any \((t,\mu,x) \in [0,T] \times \PPtwo \times \R^d\). Then, \(\psi\) is a test function
		for the parabolic HJB equation \eqref{ParabolicHJB}. Now fix a point \((\bar{t},\bar{\mu}) \in [0,T)\times \PPtwo \) and assume that \((U_T-\psi)(\bar{t},\bar{\mu})=0\) and \((\bar{t},\bar{\mu})\) is a point of maximum for the function \(U_T-\psi\), i.e.\
		\[ U_T(\bar{t},\bar{\mu})-\psi(\bar{t},\bar{\mu}) \geq U_T(t,\mu)-\psi(t,\mu) \qquad \forall t \in (0,T), \mu \in \PPtwo.  \]
		Multiplying by \(e^{\beta \bar{t}}\) and choosing \(t=\bar{t}\), we deduce that
		\[  u(\bar{\mu}) - e^{\beta \bar{t}}\psi(\bar{t},\bar{\mu}) \geq u(\mu) - e^{\beta \bar{t}}  \psi(\bar{t},\mu) \quad \forall \mu \in \PPtwo. \]
		Call \( \phi(\mu):= e^{\beta \bar{t}}\psi(\bar{t},\mu) \); notice that it is a function belonging to
		\(\TestFcts(\R^d)\) such that \(u-\phi\) has a point of maximum in \(\bar{\mu}\) and \((u-\phi)(\bar{\mu}) =0\).
		Thus, \(\phi\) is a test function for the subsolution property of the elliptic Hamilton-Jacobi-Bellman equation \eqref{HJB}. 
		
		By hypothesis, \(u\) is a subsolution to the elliptic Hamilton-Jacobi-Bellman equation \eqref{HJB}, which means that the following inequality is satisfied:
		\begin{align*}
			-\beta u(\bar{\mu}) &+ \int_{\R^d} \sup_{a \in A} \bigg( f(x,\bar{\mu},a) + <b(x,\bar{\mu},a), \partial_{\mu} \phi(\bar{\mu})(x)> \\ &+ \frac{1}{2} tr(\sigma (x,a) \sigma^T(x,a) \partial_{x} \partial_{\mu} \phi(\bar{\mu}) (x)) \bigg) \bar{\mu}(dx) \geq 0.
		\end{align*}
		If we multiply by \(e^{-\beta \bar{t}}\), by the definition of \(\phi\) we deduce that
		\begin{align*}
			0 &\leq -\beta e^{-\beta \bar{t}} u(\bar{\mu})+\int_{ \R^d }\sup_{a \in A} \bigg( f(x,\bar{\mu},a) e^{-\beta \bar{t}} + <b(x,\bar{\mu},a), e^{-\beta \bar{t}} \partial_{\mu} \phi(\bar{\mu})(x)> \\
			& + \frac{1}{2} tr(\sigma (x,a) \sigma^T(x,a) \partial_{x} \partial_{\mu} e^{-\beta \bar{t}} \phi(\bar{\mu}) (x)) \bigg) \bar{\mu}(dx) \\
			&= -\beta U_T(\bar{t}, \bar{\mu}) + \int_{ \R^d }  \sup_{a \in A} \bigg( f(x,\bar{\mu},a) e^{-\beta \bar{t}} + <b(x,\bar{\mu},a), \partial_{\mu} \psi(\bar{t},\bar{\mu})(x)> \\
			& + \frac{1}{2} tr(\sigma (x,a) \sigma^T(x,a) \partial_{x} \partial_{\mu} \psi(\bar{t}, \bar{\mu})(x)) \bigg) \bar{\mu}(dx) \\
			&=  \partial_{t} U_T(\bar{t},\bar{\mu}) + \int_{ \R^d } \sup_{a \in A} \bigg( f(x,\bar{\mu},a) e^{-\beta \bar{t}} + <b(x,\bar{\mu},a), \partial_{\mu} \psi(\bar{t},\bar{\mu})(x)> \\
			& + \frac{1}{2} tr(\sigma (x,a) \sigma^T(x,a) \partial_{x} \partial_{\mu} \psi(\bar{t}, \bar{\mu})(x)) \bigg) \bar{\mu}(dx) \\
			&=  \partial_{t} \psi(\bar{t},\bar{\mu}) + \int_{ \R^d } \sup_{a \in A} \bigg( f(x,\bar{\mu},a) e^{-\beta \bar{t}} + <b(x,\bar{\mu},a), \partial_{\mu} \psi(\bar{t},\bar{\mu})(x)> \\
			& + \frac{1}{2} tr(\sigma (x,a) \sigma^T(x,a) \partial_{x} \partial_{\mu} \psi(\bar{t}, \bar{\mu})(x)) \bigg) \bar{\mu}(dx),
		\end{align*}
		where in the last equality we used the fact that \((\bar{t}, \bar{\mu})\) is a point of maximum for the function \(U_T-\psi\), which is of class \(\Cont^1\) with respect to the time variable.
		By the arbitrariness of \((\bar{t}, \bar{\mu})\) and \(\psi\), \(U_T\) is a viscosity subsolution to the parabolic Hamilton Jacobi Bellman equation \eqref{ParabolicHJB}. As a consequence, by \eqref{ProofUniquenessSubsolution}, \( U_T(t,\mu) \leq v_T(t,\mu) \) for all \(t \in [0,T], \, \mu \in \PPtwo.\)
		
		Choosing \(t=0\), we obtain
		\begin{equation*}
			u(\mu) = U_T(0,\mu) \leq v_T(0,\mu) \qquad \forall \mu \in \PPtwo;
		\end{equation*}
		taking the limit for \(T \to + \infty\) and using the convergence of \(v_T(0,\cdot)\) to \(v\), we conclude that
		\[ u(\mu) \leq v(\mu) \qquad \forall \mu \in \PPtwo.  \]
		
		Repeating a similar reasoning for supersolutions, where some modifications are needed as the class of test functions is defined over \(\PPtwoA\), we conclude that the thesis holds.
	\end{proof}
	
	\begin{Osservazione}
		The proof of our uniqueness result relies on Theorem 4.3 in \cite{BayraktarCheungTaiQiu}. Nonetheless, the same procedure could be applied in any case  where a comparison principle for parabolic Hamilton-Jacobi-Bellman equations on the space \(\PPtwo\) is available. Moreover, the reasoning does not depend on the specific choice of the definition of viscosity supersolution: it would still hold for a definition à la Crandall-Lions, provided that a uniqueness result applies for the analogue definition of supersolution to a parabolic HJB. \qed
	\end{Osservazione}
	\section*{Acknowledgements}
	The author is a member of INDAM-GNAMPA. The author would like to thank Professor Marco Fuhrman for the several useful discussions.
	\newpage
	\appendix
	\section{Time Invariance Property: technical results}
	\label{AppendixEquaivalenctFormulation}
	We first prove the well-posedness of the ``disintegration system'' of the Mckean-Vlasov equation with respect to the first part of the Brownian trajectory, i.e.\ Lemma \ref{LemmaWellPosednessofDisintegrationEquation}. Notice that the system \eqref{DisintegrationEquation} is not a system of classical SDEs or a system of McKean-Vlasov SDEs; it is a system of SDEs parametrised by \(w \in \Cont([0,r])\), which does not belong to a finite or countable set of indices, where the ``particles'' \(X^w\) interact through the term
	\[ \int_{\Cont([0,r])} \Prob_{X_s^{w'}} W_{[0,r]}(w').  \]

	\begin{proof}[Proof of Lemma \ref{LemmaWellPosednessofDisintegrationEquation}]
		It is sufficient to prove the result on a bounded interval \([r,T]\) for any \(T>r\) and then to ``glue'' the solutions to construct the unique solution on the interval \([0, +\infty)\) with a standard procedure.
		
		Fix any \(T>r\). For any flow of probability measures \(\nu=(\nu_s)_{s \in [r,T]} \in \Cont([r,T]; \PPtwo)\), we can define the system of SDEs
		\begin{equation}
			\label{DisintegrationEquationUndernu}
			\begin{cases}
				dX_s^{w,\nu}=b(X_s^{w,\nu}, \nu_s, \alpha_s^r(w)) ds + \sigma(X_s^{w,\nu}, \nu_s, \alpha_s^r(w)) dB_s, \quad &s \in [r,T] \\
				X_r^{w,\nu}=\xi &
			\end{cases}
		\end{equation}
		for every \(w \in \Cont([0,r])\).
		For every fixed \(\nu\), \eqref{DisintegrationEquationUndernu} is a system of decoupled SDEs with random initial condition, each of which indexed by a different parameter \(w \in \Cont([0,r])\). Under assumptions (\nameref{HpCoefficients}), there thus exists a unique solution to \eqref{DisintegrationEquationUndernu}, denoted by \(X^{w,\nu}\) and such that \(\Prob\)-a.s.
		\[ X_s^{w,\nu}= \xi + \int_{r}^{s} b(X_u^{w,\nu}, \nu_u, \alpha_u^r(w)) du + \int_{r}^{s} \sigma(X_u^{w,\nu}, \nu_u, \alpha_u^r(w)) dB_u \]
		for every \(s \geq r\) and for every \(w \in \Cont([0,r])\). 
		
		We now define a map
		\begin{align*}
			\Phi: \, &\Cont([r,T]; \PPtwo) \rightarrow \Cont([r,T]; \PPtwo) \\
			& \nu= (\nu_s)_{s \geq 0} \rightarrow (\Phi(\nu)(s))_{s \geq 0}
		\end{align*}
		such that 
		\[ \Phi(\nu)(s)(E) := \int_{\Cont([0,r])} \Prob_{X_s^{w,\nu}}(E) W_{[0,r]}(dw) \quad \text{ for all } E \in \Bor(\R^d) \]
		for every \(s \geq r\). Notice that \(\Phi\) is well defined because, for every fixed \(E \in \Bor(\R^d)\) and \(s \geq r\), the map \(w \rightarrow \Prob_{X_s^{w,\nu}}(E)\) is in \(L^1(\Cont([0,r]),\Bor(\Cont([0,r])),W_{[0,r]})\) by Fubini theorem and Section 12 in \cite{StrickerYorParamètre}, which guarantees the joint measurability of \(X^{w,\nu}\) with respect to \((s,\omega,w)\). 
		Clearly, there exists a unique solution to the system \eqref{DisintegrationEquation} if and only if the map \(\Phi\) admits a unique fixed point \(\hat{\nu}\).
		
		We endow the space \(\Cont([r,T]; \PPtwo)\) with the uniform-in-time Wasserstein distance, i.e.\ for any  \(\mu, \nu \in \Cont([r,T]; \PPtwo)\)
		\[ d(\mu,\nu):= \sup \limits_{r \leq s \leq T} W_2(\mu_s, \nu_s). \] 
		Consider two flows of probability measures \(\mu, \nu \in \Cont([r,T]; \PPtwo)\). Compute
		\begin{align}
			\label{ProofPhiContraction1}
			&\sup \limits_{r \leq s \leq T} W_2^2(\Phi(\mu)(s), \Phi(\nu)(s)) \nonumber \\
			= &\sup \limits_{r \leq s \leq T} W_2^2 \bigg( \int_{\Cont([0,r])} \Prob_{X_s^{w,\mu}} W_{[0,r]} (dw), \int_{\Cont([0,r])} \Prob_{X_s^{w,\nu}} W_{[0,r]} (dw) \bigg).
		\end{align}
		We estimate the right hand side of equation \eqref{ProofPhiContraction1} by means of the Kantorovich Duality Theorem (see e.g.\ Theorem 5.10 in \cite{Villani}):
		\begin{align}
			&W_2^2 \bigg( \int_{\Cont([0,r])} \Prob_{X_s^{w,\mu}} W_{[0,r]} (dw), \int_{\Cont([0,r])} \Prob_{X_s^{w,\nu}} W_{[0,r]} (dw) \bigg) \nonumber \\
			&= \sup \limits_{ \substack{\phi,\psi \in \Cont_b(\R^d;\R): \\ \phi(x)+\psi(y) \leq |x-y|^2 } } \bigg[ \int_{\R^d} \phi(x) \bigg( \int_{\Cont([0,r])} \Prob_{X_s^{w,\mu}} W_{[0,r]} (dw) \bigg) (dx) + \nonumber \\
			&\qquad \qquad \qquad \qquad \int_{\R^d} \psi(y)
			\bigg(\int_{\Cont([0,r])} \Prob_{X_s^{w,\nu}} W_{[0,r]} (dw)\bigg) (dy) \bigg] \, \text{ (by Fubini theorem)} \nonumber\\
			&= \sup \limits_{ \substack{\phi,\psi \in \Cont_b(\R^d;\R): \\ \phi(x)+\psi(y) \leq |x-y|^2 } }  \int_{\Cont([0,r])} \bigg[ \int_{\R^d} \phi(x) \Prob_{X_s^{w,\mu}} (dx) + \int_{\R^d} \psi(y) \Prob_{X_s^{w,\nu}}(dy) \bigg] W_{[0,r]} (dw) \nonumber \\
			&\leq \int_{\Cont([0,r])} W_2^2(\Prob_{X_s^{w,\mu}}, \Prob_{X_s^{w,\nu}}) W_{[0,r]} (dw) \nonumber \\
			&\leq \int_{\Cont([0,r])} \ExpVal \bigg[ |X_s^{w,\mu}-X_s^{w,\nu}|^2 \bigg] W_{[0,r]}(dw),
			\label{ProofPhiContraction2}
		\end{align}
		where the second-to-last inequality follows applying again Kantorovich duality theorem. Putting equation \eqref{ProofPhiContraction1} and inequality \eqref{ProofPhiContraction2} together, we can say that
		\begin{align}
			\label{ProofPhiContraction3}
			\sup \limits_{r \leq s \leq T} W_2^2(\Phi(\mu)(s), \Phi(\nu)(s)) \leq \sup \limits_{r \leq s \leq T} \int_{\Cont([0,r])} \ExpVal \bigg[  |X_s^{w,\mu}-X_s^{w,\nu}|^2 \bigg] W_{[0,r]}(dw)
		\end{align}
		Proceeding as in the proof of Theorem 4.21 in \cite{CarmonaDelarue1}, we can show that, for \(C=4L^2(T+4)\),
		\begin{equation}
			\ExpVal \bigg[ |X_s^{w,\mu}-X_s^{w,\nu}|^2 \bigg] \leq  C e^{C(s-r)} \int_{r}^{s} W_2^2(\nu_t,\mu_t) dt \quad \text{for all } s \in [r,T].
			\label{ProofPhiContraction4}
		\end{equation}
		Pulling inequality \eqref{ProofPhiContraction4} back into \eqref{ProofPhiContraction3}, we can deduce that there exists a constant \(K(T) >0\) such that
		\begin{equation}
			\label{ProofPhiContraction5}
			\sup \limits_{r \leq s \leq T} W_2^2(\Phi(\mu)(s), \Phi(\nu)(s)) \leq K(T) \int_{r}^{T} W_2^2(\nu_t,\mu_t) dt.
		\end{equation}

		Now denote by \(\Phi^n\) the composition of \(n\) times the map \(\Phi\). Iterating \eqref{ProofPhiContraction5}, we conclude that
		\begin{align*}
			\sup \limits_{t \leq s \leq T} W_2^2(\Phi^n(\mu)(s), \Phi^n(\nu)(s)) &\leq K^n(T) \int_{r}^{T} \frac{(T-s)^{n-1}}{(n-1)!} W_2^2(\mu_s,\nu_s) ds \\
			& \leq K^n(T) \frac{(T-r)^n}{n!} \sup \limits_{r \leq s \leq T} W_2^2(\mu_s,\nu_s).
		\end{align*}
		For sufficiently large \(n \in \N\), \(K^n(T) \frac{(T-r)^n}{n!} < 1\) and \(\Phi^n\) is a contraction.
		
		\(\Cont([r,T]; \PPtwo)\) endowed with the uniform-in-time Wasserstein distance is a non-empty complete metric space; therefore, by Banach-Caccioppoli theorem, there exists a unique fixed point for the map \(\Phi^n\). Therefore, 
		\(\Phi\) admits a unique fixed point and the thesis holds.
	\end{proof}
	
	We conclude the Appendix with the proof of the continuity of \(J\) with respect to time uniformly with respect to the control parameter (Theorem \ref{ThmContinuityofJwrtTime}).
	\begin{proof}[Proof of Theorem \ref{ThmContinuityofJwrtTime}]
		Fix any \(\alpha \in \Adm\), \(\xi \in L^2(\Omega, \GG, \Prob)\) and \(\epsilon>0\). Compute, for \(u > t\),
		\begin{equation}
			\label{BoundJ}
			|J(u,\xi,\alpha)-J(t,\xi,\alpha)|
			\leq I_1 + I_2
		\end{equation}
		with
		\begin{align}
			\label{QuantityI1}
			&I_1:= \ExpVal \bigg[ \int_{u}^{+ \infty} |e^{-\beta (s-u)} f(X_s^{u, \xi, \alpha}, \Prob_{X_s^{u, \xi, \alpha}}, \alpha_s)- e^{-\beta (s-t)} f(X_s^{t, \xi, \alpha}, \Prob_{X_s^{t, \xi, \alpha}}, \alpha_s) | \bigg] ds \\
			\label{QuantityI2}
			& I_2:=\ExpVal \bigg[ \int_{t}^{u} e^{-\beta (s-t)} |f(X_s^{t, \xi, \alpha}, \Prob_{X_s^{t, \xi, \alpha}}, \alpha_s)| ds \bigg].
		\end{align}
		We firstly consider \(I_2\). Under hypotheses (\nameref{HpGain}), using the estimate \eqref{InitialConditionEstimate}
		\begin{align}
			\label{BoundI2}
			I_2 &\leq K \int_{t}^{u} e^{-\beta (s-t)} \ExpVal[ 1 + |X_s^{t,\xi,\alpha}|^2 + ||\Prob_{X_s^{t,\xi,\alpha}}||_2^2 ] ds \nonumber \\
			&\leq \frac{K}{\beta}[1-e^{-\beta(u-t)}] + 2K  \bigg(|| \xi ||_{L^2}^2 + 6M^2 (u-t) \bigg)(u-t) e^{(12M^2 +1-\beta)(u-t)} < \frac{\epsilon}{2}
		\end{align}
		for \(u-t < \zeta_1(\epsilon,M,K,\beta,t,\xi)\) independent of \(\alpha\), provided that \(\beta>12M^2+1\), which is guaranteed by (\nameref{HpGain}).
		We now study \(I_1\). Notice that, for any choice of \(T>0\), we can write
		\begin{equation}
			\label{BoundI1}
			I_1 =I_T + R_T,
		\end{equation}
		where
		\[ I_T= \ExpVal \bigg[ \int_{u}^{T} |e^{-\beta (s-u)} f(X_s^{u, \xi, \alpha}, \Prob_{X_s^{u, \xi, \alpha}}, \alpha_s)- e^{-\beta (s-t)} f(X_s^{t, \xi, \alpha}, \Prob_{X_s^{t, \xi, \alpha}}, \alpha_s) | ds \bigg] \]
		and
		\[ R_T=\ExpVal \bigg[ \int_{T}^{+ \infty} |e^{-\beta (s-u)} f(X_s^{u, \xi, \alpha}, \Prob_{X_s^{u, \xi, \alpha}}, \alpha_s)- e^{-\beta (s-t)} f(X_s^{t, \xi, \alpha}, \Prob_{X_s^{t, \xi, \alpha}}, \alpha_s) | ds \bigg]. \]
		
		Consider the remainder \(R_T\). Notice that, under hypotheses (\nameref{HpGain}), using again estimate \eqref{InitialConditionEstimate}
		\begin{align}
			\label{BoundRT}
			R_T &\leq \int_{T}^{+ \infty} e^{-\beta (s-u)} (1+ 2 \ExpVal[|X_s^{u, \xi, \alpha}|^2]) ds + \int_{T}^{+ \infty} e^{-\beta (s-t)} (1+2\ExpVal[|X_s^{t, \xi, \alpha}|^2]) \nonumber \\
			&\leq \int_{T}^{+ \infty} e^{-\beta (s-u)} ds + 2\int_{T}^{+ \infty} e^{(12M^2 +1-\beta) (s-u)} \bigg(|| \xi ||_{L^2}^2 + 6M^2 (s-u) \bigg) ds \nonumber \\
			&+\int_{T}^{+ \infty} e^{-\beta (s-t)} ds + 2 \int_{T}^{+ \infty} e^{(12M^2 +1-\beta) (s-t)} \bigg(|| \xi ||_{L^2}^2 + 6M^2 (s-t) \bigg) ds \nonumber \\
			&\leq 2 e^{\beta(t+1)} \int_{T}^{+ \infty} e^{-\beta s} ds + 4 e^{\beta(t+1)} \int_{T}^{+ \infty} e^{(12M^2 +1-\beta) s} \bigg(|| \xi ||_{L^2}^2 + 6M^2 (s-t) \bigg) ds \nonumber \\ 
			&< \frac{\epsilon}{4}
		\end{align}
		provided that \(u-t \leq 1\) and \(T=T(\epsilon)\) (independent of \(\alpha\)) is sufficiently large, as all the terms appearing in the previous inequality are remainders of convergent integrals.
		As equality \eqref{BoundI1} holds for any choice of \(T\), we fix \(T=T(\epsilon)\).
		
		We now study \(I_T\). Firstly, notice that, under hypotheses (\nameref{HpContf}), the function
		\[\psi:(r,x,\mu,a) \rightarrow e^{\beta r} f(x,\mu,a)  \]
		is locally H\"older continuous in \((x,\mu)\) and Lipschitz in \(r\) with Lipschitz constant depending on \((x,\mu)\), uniformly with respect to \(a\); namely, there exists a constant \(H' >0\) such that, for any \(r,r' \in [u,T]\), \(x,x' \in \R^d\), \(\mu,\mu' \in \PPtwo\) and \(a \in A\)
		\begin{align}
			\label{Continuityofpsi}
			|\psi(r,x,\mu,a)-\psi(r',x',\mu',a)|
			\leq H' [&|x-x'|^{\gamma_1}(1+|x|+|x'|)^{2-\gamma_1} \nonumber \\
			+ &W_2(\mu,\mu')^{\gamma_2} (1+||\mu||_2+||\mu'||_2)^{2-\gamma_2} \nonumber \\
			+ &|r-r'|(1+|x'|^2+||\mu'||_2^2)].
		\end{align}
		
		By estimate \eqref{InitialConditionEstimate} restricted to the interval \([0,T]\), we can easily deduce that
		\begin{equation}
			\label{InitialConditionEstimateonuT}
			\ExpVal[| X_r^{t, \xi, \alpha}|^2] \leq \bigg(|| \xi ||_{L^2}^2 + 6M^2 T\bigg) e^{(12M^2 +1)T}:=h(T,M,\xi),
		\end{equation}
		and that the same estimate holds for \(\ExpVal[| X_r^{u, \xi, \alpha}|^2]\).
		Furthermore, by standard estimates, we can prove the following continuity condition of the solution to a McKean-Vlasov SDE with respect to the initial time on a bounded time interval: there exists a constant \(K_0(T)\) such that, for any \(0 \leq t < u\), \(\xi \in L^2(\Omega, \FF_t, \Prob)\) and \(\alpha \in \Adm\),
		\begin{equation}
			\label{EstimateDifferentInitialTimes}
			\ExpVal \bigg[ |X_s^{t, \xi, \alpha}-X_s^{u, \xi, \alpha} |^2 \bigg] \leq K_0(T) (u-t) (1 + ||\xi||_{L^2}^2)
		\end{equation}
		for all \(s \in [t,T]\).
		Therefore,
		\begin{align}
			\label{BoundIT}
			I_T 
			&= \ExpVal \bigg[  \int_{u}^{T} e^{-\beta s}|\psi(u,X_s^{u, \xi, \alpha}, \Prob_{X_s^{u, \xi, \alpha}}, \alpha_s)- \psi(t,X_s^{t, \xi, \alpha}, \Prob_{X_s^{t, \xi, \alpha}}, \alpha_s) | ds \bigg] \nonumber \\
			&\leq \ExpVal \bigg[  \int_{u}^{T} H' \bigg(|X_s^{u, \xi, \alpha}-X_s^{t, \xi, \alpha}|^{\gamma_1}(1 + |X_s^{u, \xi, \alpha}| + |X_s^{t, \xi, \alpha}|)^{2-\gamma_1} \bigg) ds \bigg] \nonumber \\
			&+ H' \ExpVal \bigg[  \int_{u}^{T} \bigg( W_2(\Prob_{X_s^{u, \xi, \alpha}}, \Prob_{X_s^{t, \xi, \alpha}})^{\gamma_2} (1+ ||\Prob_{X_s^{u, \xi, \alpha}}||_2 + ||\Prob_{X_s^{t, \xi, \alpha}}||_2)^{2-\gamma_2} \bigg) ds \bigg] \nonumber \\
			&+ H' \ExpVal \bigg[  \int_{u}^{T} (u-t) (1+|X_s^{t,\xi,\alpha}|^2+||\Prob_{X_s^{t,\xi,\alpha}}||_2^2) ds \bigg] \nonumber \quad \text{(H\"older inequality)} \nonumber \\
			& \leq \int_{u}^{T} H' \bigg(\ExpVal \bigg[ |X_s^{u, \xi, \alpha}-X_s^{t, \xi, \alpha}|^2 \bigg] \bigg)^{\frac{\gamma_1}{2}} \bigg(\ExpVal \bigg[ (1+ |X_s^{u, \xi, \alpha}|+|X_s^{t, \xi, \alpha}|)^2 \bigg] \bigg)^{\frac{2-\gamma_1}{2}} ds \nonumber \\
			&+ H' \int_{u}^{T} \bigg( \ExpVal\bigg[|X_s^{u, \xi, \alpha}-X_s^{t, \xi, \alpha}|^2 \bigg] \bigg)^{\gamma_2} (1+ \ExpVal^{\frac{1}{2}}[|X_s^{u, \xi, \alpha}|^2 ] + \ExpVal^{\frac{1}{2}}[|X_s^{t, \xi, \alpha}|^2 ])^{2-\gamma_2} ds \nonumber \\
			&+ H'(u-t)(T-u) + H'(u-t) \int_{u}^{T} 2 \ExpVal \bigg[|X_s^{t,\xi,\alpha}|^2\bigg] ds \quad \nonumber \text{(by \eqref{EstimateDifferentInitialTimes} and \eqref{InitialConditionEstimateonuT})} \nonumber \\
			&\leq H' \bigg(K_0(T)(1 + ||\xi||_{L^2}^2)(u-t)\bigg)^{\frac{\gamma_1}{2}} (3 + 6 h(T,M,\xi) )^{\frac{2-\gamma_1}{2}} (T-u) \nonumber \\
			&+ H' \bigg(K_0(T)(1 + ||\xi||_{L^2}^2)(u-t)\bigg)^{\gamma_2} (1 + 2 h(T,M,\xi)^{\frac{1}{2}})^{2-\gamma_2} (T-u) \nonumber \\
			&+H'(u-t) (T+2h(T,M,\xi)) < \frac{\epsilon}{4}
		\end{align}
		for \(u-t< \zeta_2(\epsilon, t,\xi,M,T,H')\) independent of \(\alpha\).
		
		The thesis then follows plugging \eqref{BoundRT}, \eqref{BoundIT} into \eqref{BoundI1} and then \eqref{BoundI1} and \eqref{BoundI2} into \eqref{BoundJ} with \(u \in (t, t+\zeta)\) and \(\zeta=\zeta_{\epsilon,t,\xi}=\min \{1,\zeta_1,\zeta_2\}\) independent of \(\alpha \in \Adm\).
	\end{proof}
	
	\begin{Osservazione}
		Notice that the proof can be repeated with \(f\) and \(\beta\) satisfying (\nameref{HpGainAlternative}) instead of (\nameref{HpGain}). In this case, we would have the following bounds. For the integral \(I_2\), \begin{equation}
			I_2 \leq C_f \int_{t}^{u} e^{-\beta (s-t)} ds \leq C_f (u-t) < \frac{\epsilon}{2}
		\end{equation}
		provided that \(t \leq u \leq t+{\frac{\epsilon}{2C_f}}\). For the remainder \(R_T\),
		\begin{equation}
			R_T \leq  2C_f e^{\beta (1+t)}  \int_{T}^{+ \infty} e^{-\beta s} ds < \frac{\epsilon}{4}
		\end{equation}
		provided that \(u-t< \min \{1, \frac{\epsilon}{2C_f}\}\)
		and \(T=T(\epsilon)\) is sufficiently large. Furthermore, the function \(\psi\) is locally H\"older continuous in \((x,\mu)\) and Lipschitz in \(r\) uniformly with respect to \(a\), i.e.\ there exists a constant \(H'>0\) such that
		\begin{align*}
			|\psi(r,x,\mu,a)-\psi(r',x',\mu',a)|
			\leq H' [&|x-x'|^{\gamma_1}(1+|x|+|x'|)^{2-\gamma_1} \\
			+ &W_2(\mu,\mu')^{\gamma_2}(1+||\mu||_2+||\mu'||_2)^{2-\gamma_2} + |r-r'|].
		\end{align*}
		The bound for \(I_T\) thus becomes
		\begin{align*}
			I_T
			&= \ExpVal \bigg[  \int_{u}^{T} e^{-\beta s} |\psi(u,X_s^{u, \xi, \alpha}, \Prob_{X_s^{u, \xi, \alpha}}, \alpha_s)- \psi(t,X_s^{t, \xi, \alpha}, \Prob_{X_s^{t, \xi, \alpha}}, \alpha_s) | ds \bigg] \\
			&\leq \ExpVal \bigg[  \int_{u}^{T} H' \bigg(|X_s^{u, \xi, \alpha}-X_s^{t, \xi, \alpha}|^{\gamma_1}(1 + |X_s^{u, \xi, \alpha}| + |X_s^{t, \xi, \alpha}|)^{2-\gamma_1} \bigg) ds \bigg] \nonumber \\
			&+ H' \ExpVal \bigg[  \int_{u}^{T} \bigg( W_2(\Prob_{X_s^{u, \xi, \alpha}}, \Prob_{X_s^{t, \xi, \alpha}})^{\gamma_2} (1+ ||\Prob_{X_s^{u, \xi, \alpha}}||_2 + ||\Prob_{X_s^{t, \xi, \alpha}}||_2)^{2-\gamma_2} \bigg) ds \bigg] \nonumber \\
			&+ H' \ExpVal \bigg[  \int_{u}^{T} (u-t) ds \bigg] \nonumber \, (\text{H\"older inequality and \eqref{EstimateDifferentInitialTimes} and \eqref{InitialConditionEstimateonuT}} ) \\
			&\leq \int_{u}^{T} H' \bigg(K_0(T)(1 + ||\xi||_{L^2}^2)(u-t)\bigg)^{\frac{\gamma_1}{2}} (3 + 6 h(T,M,\xi) )^{\frac{2-\gamma_1}{2}} ds \nonumber \\
			&+ \int_{u}^{T} H' \bigg(K_0(T)(1 + ||\xi||_{L^2}^2)(u-t)\bigg)^{\gamma_2} (1 + 2 h(T,M,\xi)^{\frac{1}{2}})^{2-\gamma_2} ds \nonumber \\
			&+H'T(u-t) < \frac{\epsilon}{4}
		\end{align*}
		provided that \(u-t < \zeta_3(\epsilon,T,M,\xi,t,H',K_0)\) independent of \(\alpha\). Putting the three bounds together, \(|J(u,\xi,\alpha)-J(t,\xi,\alpha)|<\epsilon\) for \(u \in (t,t+\zeta)\) and \(\zeta=\zeta_{\epsilon,t,\xi}= \min \{\zeta_3, 1, \frac{\epsilon}{2C_f}\}\) independent of \(\alpha\). \qed
	\end{Osservazione}
	
	\bibliographystyle{plain}
	\bibliography{References}
	
\end{document}